\documentclass[11pt]{article}
\usepackage[letterpaper]{geometry}
\usepackage{authblk}
\usepackage{amsthm,amsmath,amsfonts}
\usepackage{color,graphicx}
\usepackage[colorlinks,linkcolor=blue,citecolor=cyan,anchorcolor=blue]{hyperref}
\usepackage[numbers,sort]{natbib}
\usepackage{enumerate}
\usepackage{wrapfig}
\usepackage{accents}
\usepackage{tikz}
\usetikzlibrary{arrows,
                petri,
                topaths}
\usepackage{tkz-berge}
\usepackage{verbatim}
\usetikzlibrary{calc,arrows.meta,positioning}
\usetikzlibrary{plotmarks}

\newcommand{\R}{\mathbb{R}}

\newcommand{\Z}{\mathbb{Z}}
\newcommand{\E}{\mathbb{E}}

\newcommand{\Prob}{\mathbb{P}}

\providecommand{\abs}[1]{\left\lvert#1\right\rvert}

\usepackage[normalem]{ulem}

\newtheorem{lemma}{Lemma}
\newtheorem{theorem}{Theorem}
\newtheorem{corollary}{Corollary}
\theoremstyle{definition}
\numberwithin{equation}{section}

\newtheorem{proposition}{Proposition}
\newtheorem{remark}{Remark}
\begin{document}
\title{Steady-state analysis of the Join the Shortest Queue model in the Halfin-Whitt regime}
\author[1]{Anton Braverman}
\affil[1]{Kellogg School of Management at Northwestern University}
\maketitle
\abstract{This paper studies the steady-state properties of the Join the Shortest Queue model in the Halfin-Whitt regime. 
We focus on the process tracking the number of idle servers, and the number of servers with non-empty buffers. Recently,  \cite{GamaEsch2015} proved that a scaled version of this process converges, over finite time intervals, to a two-dimensional diffusion limit as the number of servers goes to infinity. In this paper we prove that the diffusion limit is exponentially ergodic, and that the diffusion scaled sequence of the steady-state number of idle servers and non-empty buffers is tight. Combined with the process-level convergence proved in \cite{GamaEsch2015}, our results imply convergence of steady-state distributions. The methodology used is the generator expansion framework based on Stein's method, also referred to as the drift-based fluid limit Lyapunov function approach in \cite{Stol2015}. One technical contribution to the framework is to show how it can be used as a general tool to establish exponential ergodicity.}

\section{Introduction.} \label{sec:intro}
We consider a system with $n$ identical servers, where customers arrive according to a Poisson process with rate $n\lambda$, and service times are i.i.d.\ exponentially distributed with rate $1$. Each server maintains an individual buffer of infinite length. When a customer arrives, he will either enter service immediately if an idle server is available, or be routed to the server with the smallest number of customers in its buffer; ties are broken arbitrarily. Once a customer is routed to a server, he cannot switch to a different server. This model is known as the Join the Shortest Queue (JSQ) model. To describe the system, let $Q_i(t)$ be the number of servers with $i$ or more customers at time $t \geq 0$, and let $Q(t) = (Q_{i}(t) )_{i=1}^{\infty}$. Then $\{Q(t)\}_{t \geq 0}$ is a continuous time Markov chain (CTMC), and it is positive recurrent provided $n \lambda < 1$ \cite{Bram2011a}. Let  $Q_i$  be the random variables having the stationary distributions of $\{Q_i(t)\}_{t}$.

In this paper we work in the Halfin-Whitt regime \cite{HalfWhit1981}, which assumes that
\begin{align}
\lambda = 1 - \beta/\sqrt{n}, \label{eq:hw}
\end{align}
for some fixed $\beta > 0$. The first paper to study the JSQ model in this regime is \cite{GamaEsch2015}, which shows that the scaled process 
\begin{align}
\Big\{ \Big( \frac{Q_1(t) - n}{\sqrt{n}}, \frac{Q_2(t)}{\sqrt{n}}, \frac{Q_3(t)}{\sqrt{n}}, \ldots\Big) \Big\}_{t \geq 0} \label{eq:diffscale}
\end{align}
converges to a diffusion limit as $n \to \infty$. The diffusion limit of \eqref{eq:diffscale} is essentially two dimensional, because $Q_i(t)/\sqrt{n}$ becomes negligible for $i \geq 3$. The results of \cite{GamaEsch2015} are restricted to the transient behavior of the JSQ model and steady-state convergence is not considered, i.e.\ convergence to the diffusion limit is proved only for finite time intervals. 

In the present paper, we study the steady-state properties of the JSQ system. Specifically, we prove the existence of an explicitly known constant $C(\beta) > 0$ depending only on $\beta$ such that
\begin{align}
&n - \E Q_1 = n (1 - \lambda), \notag \\
&\E Q_2 \leq C(\beta) \sqrt{n} , \notag \\
&\E Q_i \leq C(\beta), \quad i \geq 3, \quad n \geq 1.\label{eq:introresults}
\end{align}
In other words, the expected number of idle servers is known, the expected number of non-empty buffers is at most of order $\sqrt{n}$, and the expected number of buffers with two or more waiting customers is bounded by a constant independent of $n$. A consequence of \eqref{eq:introresults} is tightness of the sequence of diffusion-scaled stationary distributions. 

In addition to \eqref{eq:introresults}, we also prove that the two-dimensional diffusion limit of the JSQ model is exponentially ergodic. Stability of this diffusion limit remained an open question until the present paper. Combining the process-level convergence of \cite{GamaEsch2015}, tightness of the prelimit stationary distributions in \eqref{eq:introresults}, and stability of the diffusion limit, we are able to justify convergence of the stationary distributions via a standard limit-interchange argument.

To prove our results, we use the generator expansion framework, which is a manifestation of Stein's method \cite{Stei1986} in queueing theory and was recently introduced to the stochastic systems literature in \cite{Gurv2014, BravDai2017};  see \cite{BravDaiFeng2016} for an accessible introduction. The idea is to perform Taylor expansion on the generator of a CTMC, and by looking at the second-order terms,  to identify a diffusion model approximating the CTMC. One then proves bounds on the steady-state approximation error of the diffusion, which commonly results in convergence rates to the diffusion approximation \cite{Gurv2014, BravDai2017, BravDaiFeng2016, GurvHuan2016, FengShi2017}. In this paper, we use only the first-order terms of the generator expansion, which correspond to the generator of a related fluid model. We then carry out the machinery of Stein's method to prove convergence rates to the fluid model equilibrium. The bounds in \eqref{eq:introresults} are then simply an alternative interpretation of these convergence rates. For other examples of Stein's method for fluid, or mean-field models, see \cite{Ying2016,Ying2017, Gast2017,GastHoud2017}. Specifically, \cite{Ying2016} was the first to make the connection between Stein's method and convergence rates to the mean-field equilibrium.

Our approach can also be tied to the drift-based fluid limit (DFL) Lyapunov functions used in \cite{Stol2015}, which appeared a few years before \cite{Ying2016}. As we will explain in more detail in Section~\ref{sec:derbounds}, the DFL approach and Stein's method for mean-field approximations are essentially one and the same.

This paper contributes another example of the successful application of the generator expansion method to the queuing literature. Although the general framework has already been laid out in previous work, examples of applying the framework to non-trivial systems are the only way to display the power of the framework and promote its adoption in the research community. Furthermore, tractable examples help showcase and expand the versatility of the framework and the type of results it can prove. The present paper contributes from this angle in two ways. First, the JSQ model is an example where the dimension of the CTMC is greater than that of the diffusion approximation. To justify the approximation, one needs a way to show that the additional dimensions of the CTMC are asymptotically negligible; this is known as state space collapse (SSC). 
 Our way of dealing with SSC in Section~\ref{sec:tightness} differs from the typical solution of bounding the magnitude of the SSC terms \cite{BravDai2017, EryiSrik2012, MaguSrik2016,BurlMaguSrik2016} (only \cite{BravDai2017} of the aforementioned papers uses the generator expansion framework, but the rest still deal with steady-state SSC in a conceptually similar way). Second, this paper presents the first working example of the generator expansion framework being used to prove exponential ergodicity of the diffusion approximation. The insight used is simple, but can be easily generalized to prove exponential ergodicity for other models.


\subsection{Literature review and contributions.} \label{sec:lit}
Early work on the JSQ model appeared in the late 50's and early 60's \cite{Haig1958,King1961}, followed by a number of papers in the 70's--90's \cite{FlatMcKe1977a, FoscSalz1978, Half1985, WangZhan1989, HsuWangZhan1995}. This body of literature first studied the JSQ model with two servers, and later considered heavy-traffic asymptotics in the setting where the number of servers $n$ is fixed, and $\lambda \to 1$; see \cite{GamaEsch2015} for an itemized description of the aforementioned works. A more recent paper \cite{EryiSrik2012} considers the steady-state behavior of the JSQ model, but again in the setting where $n$ is fixed, and $\lambda \to 1$. 

The asymptotic regime where $n \to \infty$ has been untouched until very recently.  In \cite{Stol2015a}, the author studies a variant of the JSQ model where the routing policy is to join an idle server if one is available, and otherwise join any buffer uniformly; this is known as the Join the Idle Queue (JIQ) policy. In that paper, the arrival rate is $n \lambda$ where  $\lambda < 1$ is fixed, and $n \to \infty$. The author shows that in this underloaded asymptotic regime, JIQ is asymptotically optimal on the fluid scale, and therefore asymptotically equivalent to the JSQ policy. We have already described \cite{GamaEsch2015}, which is the first paper to study a non-underloaded regime. In \cite{MukhBorsLeeuWhit2016}, the authors work in the Halfin-Whitt regime and show that JIQ is asymptotically optimal, and therefore asymptotically equivalent to JSQ, on the diffusion scale. Most recently,  \cite{GuptWalt2017} studies the JSQ model in the non-degenerate slowdown (NDS) regime introduced in \cite{Atar2012}. In this regime, $\lambda = 1 - \beta/n$ for some fixed $\beta > 0$, i.e.\ NDS is even more heavily loaded than the Halfin-Whitt regime. The authors of \cite{GuptWalt2017} establish a diffusion limit for the total customer count process. For a recent overview of load balancing algorithms see \cite{BoorBorsLeeuMukh2017}; that paper includes the JSQ algorithm and the closely related power-of-$d$ class of policies.

In the asymptotic regime where $n \to \infty$, all previous considerations of the diffusion-scaled model   \cite{GamaEsch2015,MukhBorsLeeuWhit2016,GuptWalt2017} have been in the transient setting. In particular, convergence to the diffusion limit is only proved over finite time intervals. In contrast, the present paper deals with steady-state distributions. Since the seminal work of \cite{GamaZeev2006}, justifying convergence of steady-state distributions has become the standard in heavy-traffic approximations, and is recognized as being a non-trivial step beyond convergence over finite-time intervals \cite{BudhLee2009, ZhanZwar2008, Kats2010, YaoYe2012, Tezc2008, GamaStol2012, Gurv2014a, Stol2015, GurvHuan2016, BravDai2017,BravDaiFeng2016,Gurv2014}. 

The methodology used in this paper can be discussed in terms of \cite{Ying2016,Ying2017, Gast2017, Stol2015}. The main technical driver of our results are bounds on the derivatives of the solution to a certain first order partial differential equation (PDE) related to the fluid model of the JSQ system. In the language of \cite{Stol2015}, we need to bound the derivatives of the DFL Lyapunov function. These derivative bounds are a standard requirement to apply Stein's method, and \cite{Ying2016,Ying2017, Gast2017} provide sufficient conditions to bound these derivatives for a large class of PDEs. The bounds in \cite{Ying2016,Ying2017, Gast2017} require continuity of the vector field defining the fluid model, but the JSQ fluid model does not satisfy this continuity due to a reflecting condition at the boundary. To circumvent this, we leverage knowledge of how the fluid model behaves to give us an explicit expression for the PDE solution, and we bound its derivatives directly using this expression. Using the behavior of the fluid model is similar to what was done in \cite{Stol2015}. However, bounding the derivatives in this way requires detailed understanding of the fluid model, and as such this is a case-specific approach that varies significantly from one model to another. Furthermore, unlike \cite{Stol2015} where the dimension of the CTMC equals the dimension of the diffusion approximation, our CTMC is infinite-dimensional whereas the diffusion process is two-dimensional. These additional dimensions in the CTMC create additional technical difficulties which we handle in Section~\ref{sec:tightness}.

Regarding our proof of exponential ergodicity. The idea of using a fluid model Lyapunov function to establish exponential ergodicity of the diffusion model was initially suggested in Lemma 3.1 of \cite{Gurv2014}. However, the discussion in \cite{Gurv2014} is at a conceptual level, and it is only after the working example of the present paper that we have a simple and general implementation of the idea. Indeed, our Lyapunov function in  Section~\ref{sec:proofergod} violates the condition in Lemma 3.1 of \cite{Gurv2014}.


\subsection{Notation.}
We use $\Rightarrow$ to denote weak convergence, or convergence in distribution. 
We use $1(A)$ to denote the indicator of a set $A$. We use $D = D([0,\infty),\R)$ to denote the space of right continuous functions with left limits mapping $[0,\infty)$ to $\R$. For any integer $k \geq 2$, we let $D^{k} = D([0,\infty),\R^{k})$ be the product space $D \times \ldots \times D$. Let 
\begin{align}
\Omega = (-\infty,0] \times [0,\infty). \label{eq:omega}
\end{align}
Going forward, we adopt the convention that for any function $f:\Omega \to \R$, partial derivatives are understood to be one-sided derivatives for those values $x \in \partial \Omega$ where the derivative is not defined. For example, the partial derivative with respect to  $x_1$  is not defined on the set $\{x_1=0,\ x_2\geq 0\}$. In particular, for any integer $k > 0$ we let  $C^{k}(\Omega)$ be the set of $k$-times continuously differentiable functions $f: \Omega	\to \R$ obeying the notion of one-sided differentiability  just described. We use $f_i(x)$ to denote $\frac{d f(x)}{d x_i}$. 

The rest of the paper is structured as follows. We state our main results in Section~\ref{sec:main}, and provide a roadmap to prove them in Section~\ref{sec:tightness}. Section~\ref{sec:derbounds} is devoted to understanding the JSQ fluid model, and using this to prove the derivative bounds that drive the proof of our main results.

\section{Model and main results.} \label{sec:main}

Consider the CTMC $\{Q(t)\}_{t \geq 0}$ introduced in Section~\ref{sec:intro}. The state space of the CTMC is 
\begin{align*}
S = \big\{ q \in \{ 0,1,2, \ldots , n \}^{\infty} \ |\ q_{i} \geq q_{i+1} \text{ for } i \geq 1 \text{ and } \sum_{i=0}^{\infty}q_i < \infty \big\}.
\end{align*}
The requirement that $\sum_{i=0}^{\infty}q_i < \infty$ if $q \in S$ means that we only consider states with a finite number of customers. Recall that  $Q_i$ are random variables having the stationary distributions of $\{Q_i(t)\}_{t}$, and let $Q = (Q_{i})$ be the corresponding vector. 
Let us also define the fluid-scaled CTMC $\{X(t)\}_{t\geq 0}$ by
\begin{align*}
X_1(t) = \frac{Q_1(t) - n}{n}, \quad X_i(t) = \frac{Q_i(t)}{n}, \quad i \geq 2.
\end{align*}
Also, let $X_i$ be the random variables having the stationary distributions of  $\{X_i(t)\}_{t\geq 0}$, and set $X = (X_i)_{i=1}^{\infty}$. In addition to the fluid scaling, we refer to $\{\sqrt{n} X(t)\}_{t\geq 0}$ and $\{\sqrt{n} X\}_{n=1}^{\infty}$ as the diffusion-scaled CTMC and sequence of stationary distributions, respectively.

As mentioned, convergence of the diffusion-scaled process was already proved. The following result is copied from \cite{MukhBorsLeeuWhit2016} (but it was first proved in \cite{GamaEsch2015}).
\begin{theorem}[Theorem 1 of \cite{MukhBorsLeeuWhit2016}]\label{thm:transient}
Suppose  $Y(0) = (Y_1(0),  Y_2(0)) \in \R^2$ is a random vector such that  $\sqrt{n}X_i(0) \Rightarrow Y_i(0)$ for $i = 1,2$ as $n \to \infty$ and $\sqrt{n}X_i(0) \Rightarrow 0$ for $i \geq 3$ as $n \to \infty$. Then the process $\{\sqrt{n}(X_1(t),  X_2(t))\}_{t\geq 0}$ converges uniformly over bounded intervals to $\{(Y_1(t),Y_2(t))\}_{t \geq 0} \in D^2$, which is the unique solution of the stochastic integral equation 
\begin{align}
&Y_1(t) = Y_1(0) + \sqrt{2} W(t) - \beta t + \int_{0}^{t} (-Y_1(s) + Y_2(s)) ds - U(t), \notag \\
&Y_2(t) = Y_2(0) + U(t)  - \int_{0}^{t} Y_2(s) ds, \label{eq:diffusion}
\end{align}
where $\{W(t)\}_{t\geq 0}$ is standard Brownian motion and $\{U(t)\}_{t\geq 0}$ is the unique non-decreasing, non-negative process in $D$ satisfying $\int_{0}^{\infty} 1(Y_1(t)<0) d U(t) = 0$.
\end{theorem}
Although process-level convergence to the diffusion limit was proved, the question of convergence of stationary distributions remained open until the present paper. Establishing steady-state convergence requires two ingredients: 1) proving tightness of the diffusion-scaled sequence $\{\sqrt{n} X\}_{n=1}^{\infty}$ and 2) showing that the diffusion model in \eqref{eq:diffusion} is positive recurrent. These ingredients are established via the following two theorems. 
\begin{theorem}  \label{thm:main}
For each $\beta > 0$, there exists a constant $C(\beta)$ such that for all $n \geq 1$, 
\begin{align}
\abs{\sqrt{n}X_i} \leq&\ C(\beta), \quad i =1,2, \label{eq:thm1}\\
 \abs{nX_i} \leq&\ C(\beta), \quad i \geq 3. \label{eq:thm2}
\end{align}
\end{theorem}

\begin{theorem} \label{thm:ergodic}
The diffusion process $\{(Y_1(t),Y_2(t))\}_{t \geq 0}$ defined in \eqref{eq:diffusion} is positive recurrent.
\end{theorem}
In addition to positive recurrence of the diffusion, we will actually prove that it is exponentially ergodic. Combining Theorems~\ref{thm:transient}--\ref{thm:ergodic}, we arrive at the following proposition. The proof is simple and relegated to Section~\ref{app:interchange}.
\begin{proposition} \label{thm:interchange}
Let $Y = (Y_1,Y_2)$ have the stationary distribution of the diffusion process defined in \eqref{eq:diffusion}. Then 
\begin{align}
\sqrt{n}(X_1,X_2) \Rightarrow Y \text{ as $n \to \infty$}.
\end{align}
\end{proposition}

In the remainder of this paper, we prove Theorems~\ref{thm:main} and \ref{thm:ergodic}. The former is proved in Section~\ref{sec:tightness} while the latter is proved in Section~\ref{sec:difflimit}. As we will see, both theorems will  proved using very similar methodology. Introduction of this methodology is the topic of the next section.
 

%

\section{Proving tightness} \label{sec:tightness}
Let $G_Q$ be the generator of the CTMC $\{Q(t)\}_{t \geq 0}$, which acts on function $f : S \to \R$ in the following way:
\begin{align*}
G_Q f(q) =&\ n \lambda 1(q_1 < n) \big( f(q + e^{(1)}) - f(q) \big) \\
&+ \sum_{i=2}^{\infty} n \lambda 1(q_1 = \ldots = q_{i-1} = n, q_i < n) \big( f(q+e^{(i)}) - f(q)\big)\\
&+ \sum_{i=1}^{\infty} (q_i - q_{i+1}) \big(f(q-e^{(i)}) - f(q)\big),
\end{align*}
where $e^{(i)}$ is the infinite dimensional vector where the $i$th element equals one, and the rest equal zero. The generator of the CTMC encodes the stationary behavior of the chain. The relationship between the generator and the stationary distribution can be exploited via the following lemma, which is proved in Section~\ref{app:gzlemma}.
\begin{lemma} \label{lem:gz}
For any function $f: S \to \R$ such that $\E \abs{f(Q)} < \infty$, 
\begin{align}
\E G_Q f(Q) = 0. \label{eq:bar}
\end{align}
\end{lemma}
By choosing different test functions $f(q)$, we can use \eqref{eq:bar} to obtain stationary performance measures of our CTMC. As an example of the idea, we are able to prove the following using simple test functions.
\begin{lemma} \label{lem:q1}
For any $n \geq 1$ and $\lambda \in (0,1)$, 
\begin{align}
&\E Q_1 = n\lambda, \label{eq:q1} \\
&\E Q_i = n \lambda \Prob( Q_1 = \ldots = Q_{i-1} = n), \quad i > 1. \label{eq:qi}
\end{align}
\end{lemma}
The lemma is proved in Section~\ref{app:q1}. The main idea is to apply $G_Q$ to $f(q) =  \sum_{j=i}^{\infty} q_j$ to get the equation involving $\E Q_i$. This type of analysis is also commonly called Lyapunov drift analysis, and so in this paper we will use the terms `test function' and `Lyapunov function' interchangeably. Observe that \eqref{eq:q1} implies that $n - \E Q_1 = n(1-\lambda)$, which is what was claimed in \eqref{eq:introresults} of Section~\ref{sec:intro}. Furthermore, tightness of $\{\sqrt{n} X_1\}$ follows because 
 \begin{align*}
\sqrt{n} \E \abs{X_1} = \frac{n-\E Q_1}{\sqrt{n}} = \frac{n(1-\lambda)}{\sqrt{n}} = \beta.
\end{align*}
Despite having \eqref{eq:qi} at our disposal, we do not prove the rest of Theorem~\ref{thm:main} by analyzing $\Prob( Q_1 = \ldots = Q_{i-1} = n)$ directly. This is because we do not have a good handle to control these probabilities. One may continue to experiment by applying $G_Q$ to various test functions in the hope of getting more insightful results from \eqref{eq:bar}, e.g.\ an expression for $\E Q_i$ that does not involve the complicated term $\Prob( Q_1 = \ldots = Q_{i-1} = n)$. In general, the more complicated the Markov chain, the less likely that ad-hoc experimentation with test functions will be a productive strategy. 

Let us begin to derive a more systematic approach to picking the test function. For a function $f:\R^2 \to \R$, define the lifted version $ A f: S \to \R$ by
\begin{align*}
(A f)(q) = f(x_1,x_2) = f(x), \quad q \in S,
\end{align*}
where $x_1 = (q_1-n)/n$, and $x_i = q_i/n$ for $i \geq 2$. Keeping in mind the relationship between $x$ and $q$, we will abuse notation and sometimes write $(Af)(x)$. The generator  $G_X$  of the CTMC $\{ X(t) \}_{t\geq 0}$ acts on $Af$ as follows:
\begin{align*}
G_X Af(x) =&\ n\lambda 1(q_1<n) \big(f(x+e^{(1)}/n)- f(x)\\
&+  n \lambda 1(q_1=n, q_2  < n) \big( f(x+e^{(2)}/n) - f(x)\big)\\\\
&+(q_1-q_2) \big(f(x-e^{(1)}/n) - f(x)\big) + (q_2 - q_{3}) \big(f(x-e^{(2)}/n) - f(x)\big).
\end{align*}
The essence of the generator comparison framework is that we can perform Taylor expansion on $G_X$ above to extract a `generator' for the associated fluid model. To be precise, for any differentiable function $f: \R^2 \to \R$, let us define 
\begin{align}
L f(x)  = (- x_1 + x_2 - \beta/\sqrt{n}) f_1(x) - x_2 f_2(x), \quad x \in \R^2, \label{eq:L}
\end{align}
where $f_i(x) = \frac{d f(x)}{d x_i}$. 
 The following lemma expands $G_X$ into $L$ plus additional terms, and is proved in Section~\ref{app:taylor}. Recall the set $\Omega$ introduced in \eqref{eq:omega}.
\begin{lemma} \label{lem:gentaylor}
For any $q_i \in \Z_+$, let $x_1 = (q_1-n)/n$ and  $x_i = q_i/n$ for $i \geq 2$. Suppose $f(y_1,y_2)$ is defined on $\Omega$, and   $f_1(\cdot,y_2), f_{2}(y_1,\cdot)$ are absolutely continuous for all $y \in \Omega$. Then for all $q \in S$,
\begin{align*}
G_X A f(q) =&\ Lf(x) + (f_2(x)-f_1(x) ) \lambda 1(x_1 = 0) +\varepsilon(x),
\end{align*} 
where
\begin{align*}
\varepsilon(x) =&\   - f_2(x) \lambda 1(q_1 = q_2 = n)  + q_3 \int_{x_2-1/n}^{x_2} f_2(x_1,u) du  \notag \\
&+ n\lambda 1(q_1 < n) \int_{x_1}^{x_1+1/n} (x_1 + 1/n - u) f_{11}(u,x_2) du  \notag \\
&+ n\lambda 1(q_1 = n, q_2 < n) \int_{x_2}^{x_2+1/n} (x_2 + 1/n - u) f_{22}(x_1,u) du \notag \\
&+ (q_1 - q_2) \int_{x_1-1/n}^{x_1} (u - (x_1 - 1/n)) f_{11}(u,x_2) du \notag \\
&+ q_2 \int_{x_2-1/n}^{x_2} (u - (x_2 - 1/n)) f_{22}(x_1,u) du.
\end{align*}
\end{lemma}
To make use of Lemma~\ref{lem:gentaylor}, observe that $\E \abs{f(X_1,X_2)} < \infty $ for any $f:\R^2 \to \R$ because $(X_1,X_2)$ can only take finitely many values.  A variant of Lemma~\ref{lem:gz} then tells us that 
\begin{align}
\E G_X Af(X) =0. \label{eq:liftedgx}
\end{align} 
Combining \eqref{eq:liftedgx} with the Taylor expansion in Lemma~\ref{lem:gentaylor} then yields
\begin{align}
\E G_X Af(X) = \E Lf(X) + \E \big((f_2(X)-f_1(X) ) \lambda 1(X_1 = 0)\big) +  \E \varepsilon(X)=0. \label{eq:gendetailed}
\end{align}
In other words, the expression above says that $G_X Af(X)$ can be decomposed into three parts. The first term $Lf(X)$ represents the first-order drift of the CTMC, which is commonly referred to as the fluid model of the process. The higher order terms (corresponding to the diffusion approximation) are grouped into $\varepsilon(X)$, which will act as an error term for our purposes. The final term, $(f_2(X)-f_1(X) ) \lambda 1(X_1 = 0)$, is a reflection term that is present because $X_1$ has to be non-positive. Unlike the terms in $\varepsilon(X)$, this term cannot be treated as error.

We will construct a Lyapunov function $f(x)$ such that a) $L f(x)$ is well-understood and the reflection term $(f_2(x)-f_1(x) ) \lambda 1(x_1 = 0)$ vanishes, and b) the derivatives of $f(x)$ can be controlled to bound $\E\varepsilon(X)$. The following result tells us of the existence of a Lyapunov function $f(x)$ that satisfies both a) and b).  
\begin{lemma} \label{lem:pde}
Let $L f(x)$ be as in \eqref{eq:L} and fix $\kappa > \beta$. The PDE 
\begin{align} 
L f (x) =&\ -\big((x_2 - \kappa/\sqrt{n}) \vee 0\big), \quad x \in \Omega, \label{eq:poisson1}\\
f_1 (0,x_2) =&\ f_2(0,x_2), \quad   x_2 \geq 0 \label{eq:poisson2}
\end{align} 
has a solution $f^{*}(x)$  with $f_1^*(\cdot,x_2), f_{2}^*(x_1,\cdot)$  absolutely continuous for  all $x \in \Omega$, and the second-order weak derivatives satisfy
\begin{align}
& f_{11}^{*}(x), f_{12}^{*}(x), f_{22}^{*}(x) \geq 0,& \quad x \in \Omega, \label{eq:dpos}\\
&f_{11}^{*}(x) = f_{22}^{*}(x) = 0, &\quad x_2 \in [0,\kappa/\sqrt{n}], \label{eq:dzero}\\
&f_{11}^{*}(x) \leq \frac{\sqrt{n}}{\beta}\Big( \frac{\kappa}{\kappa-\beta} + 1\Big), \quad   f_{22}^{*}(x) \leq \frac{\sqrt{n}}{\beta}\Big( 5  +\frac{2\kappa}{\kappa- \beta }   \Big), &\quad x_2 \geq \kappa/\sqrt{n}. \label{eq:d22}
\end{align}
\end{lemma}
The proof of Lemma~\ref{lem:pde} is postponed to Section~\ref{sec:derbounds}, where we also comment on the relationship of $f^*(x)$ to the DFL Lyapunov function of \cite{Stol2015}. We are now in a position to prove that $\{\sqrt{n}X_2\}$ is tight. Fix $\kappa > \beta$ and let $f^*(x)$ be as in Lemma~\ref{lem:pde}. Going back to the Taylor expansion in Lemma~\ref{lem:gentaylor}, 
\begin{align*}
 G_X Af^{*}(q) =&\  Lf^{*}(x) +\big((f_2^{*}(x)-f_1^{*}(x) ) \lambda 1(x_1 = 0)\big)+  \varepsilon(x) \\
 =&\   -\big((x_2 - \kappa/\sqrt{n}) \vee 0\big) +  \varepsilon(x), \quad q \in S.
\end{align*}
Taking expected values on both sides and applying \eqref{eq:liftedgx}, we conclude that  
\begin{align}
  \E \big((X_2 - \kappa/\sqrt{n}) \vee 0\big) = \E \varepsilon(X). \label{eq:stein}
\end{align}
We now prove that $\sqrt{n}\abs{\E \varepsilon(X)}$ is bounded by some constant $C(\beta)$. Recall 
\begin{align}
\varepsilon(x) =&\ - f_2^{*}(x) \lambda 1(q_1 = q_2 = n)  + q_3 \int_{x_2-1/n}^{x_2} f_2^{*}(x_1,u) du  \label{eq:diff2}\\
&+ n\lambda 1(q_1 < n) \int_{x_1}^{x_1+1/n} (x_1 + 1/n - u) f_{11}^{*}(u,x_2) du  \label{eq:diff3}\\
&+ n\lambda 1(q_1 = n, q_2 < n) \int_{x_2}^{x_2+1/n} (x_2 + 1/n - u) f_{22}^{*}(x_1,u) du \label{eq:diff4}\\
&+ (q_1 - q_2) \int_{x_1-1/n}^{x_1} (u - (x_1 - 1/n)) f_{11}^{*}(u,x_2) du \label{eq:diff5}\\
&+ q_2 \int_{x_2-1/n}^{x_2} (u - (x_2 - 1/n)) f_{22}^{*}(x_1,u) du. \label{eq:diff6}
\end{align}
We first argue that lines \eqref{eq:diff3}-\eqref{eq:diff6} are all non-negative and provide an upper bound for them. By \eqref{eq:dpos} we know that \eqref{eq:diff3}-\eqref{eq:diff6} all equal zero when $x_2 < \kappa/\sqrt{n} - 1/n$. 

Now suppose $x_2 \geq \kappa/\sqrt{n} - 1/n$. From \eqref{eq:d22} and the fact that $q_1\geq q_2$ if $q \in S$, we can see that each of \eqref{eq:diff3} and \eqref{eq:diff5} is non-negative and bounded by $\frac{1}{\beta\sqrt{n}}  \Big( \frac{\kappa}{\kappa-\beta} + 1\Big)$. Similarly, \eqref{eq:d22} tells us that each of \eqref{eq:diff4} and \eqref{eq:diff6} is non-negative and bounded by $\frac{1}{\beta\sqrt{n}} \Big(5+\frac{2\kappa}{\kappa-\beta}\Big)$. We conclude that \eqref{eq:diff3}-\eqref{eq:diff6} is bounded by 
\begin{align*}
\frac{1}{\beta\sqrt{n}} \Big( 12  +\frac{6\kappa}{\kappa- \beta }   \Big)1(x_2 \geq \kappa/\sqrt{n} - 1/n).
\end{align*}  
Then \eqref{eq:stein} implies
\begin{align*}
0 \leq&\ \E \big((X_2 - \kappa/\sqrt{n}) \vee 0\big) = \E \varepsilon(X) \notag \\
\leq&\  \frac{1}{\beta\sqrt{n}} \Big( 12  +\frac{6\kappa}{\kappa- \beta }   \Big)\Prob(X_2 \geq \kappa/\sqrt{n} - 1/n) \notag \\
&- f_{2}^{*}(0,1)\lambda\Prob( Q_1=Q_2=n) + \E\bigg[ Q_3 \int_{X_2-1/n}^{X_2} f_2^{*}(X_1,u) du \bigg].
\end{align*}
The term containing $Q_3$ above is present because our  CTMC is infinite dimensional, but the PDE \eqref{eq:poisson1}--\eqref{eq:poisson2} is two-dimensional. To deal with this error term, we invoke Lemma~\ref{lem:q1}:
\begin{align*}
& - f_{2}^{*}(0,1)\lambda\Prob( Q_1=Q_2=n) + \E\bigg[ Q_3 \int_{X_2-1/n}^{X_2} f_2^{*}(X_1,u) du \bigg] \\
=&\ - f_{2}^{*}(0,1)\frac{1}{n} \E Q_3 + \E\bigg[ Q_3 \int_{X_2-1/n}^{X_2} f_2^{*}(X_1,u) du \bigg] \\
=&\  \E\bigg[ Q_3 \int_{X_2-1/n}^{X_2} (f_2^{*}(X_1,u) - f_{2}^{*}(0,1)) du \bigg] \\
\leq&\ 0,
\end{align*}
where in the last inequality we used $f_{21}(x), f_{22}(x) \geq 0$ from \eqref{eq:dpos}. We conclude that 
\begin{align}
\E \big((X_2 - \kappa/\sqrt{n}) \vee 0\big) \leq \frac{1}{\beta\sqrt{n}} \Big( 12  +\frac{6\kappa}{\kappa- \beta }   \Big)\Prob(X_2 \geq \kappa/\sqrt{n} - 1/n), \label{eq:interm3}
\end{align}
and hence 
\begin{align*}
\E \sqrt{n}X_2 = \kappa  + \E (\sqrt{n}X_2 - \kappa ) \leq \kappa + \frac{1}{\beta} \Big( 12  +\frac{6\kappa}{\kappa- \beta }   \Big)\Prob(X_2 \geq \kappa/\sqrt{n} - 1/n),
\end{align*}
which establishes tightness of $\{\sqrt{n} X_2\}_{n=1}^{\infty}$.
The remainder of Theorem~\ref{thm:main}, namely \eqref{eq:thm2}, follows from a relatively simple bootstrapping argument involving \eqref{eq:interm3}. The proof is presented in Section~\ref{app:main2}. In the following section, we describe how to construct $f^*(x)$.


%
%


\section{The Lyapunov function} \label{sec:derbounds}
The aim of this section is to prove Lemma~\ref{lem:pde}. The following informal discussion provides a roadmap of the procedure. Given $x \in \Omega$, \cite[Lemma 1]{GamaEsch2015} implies the existence and uniqueness of a solution $v^{x}(t)$ to the system of integral equations
\begin{align}
&v_1(t) = x_1 - \frac{\beta}{\sqrt{n}}t - \int_{0}^{t} (v_1(s) - v_2(s)) ds - U_1(t), \notag \\
&v_2(t) = x_2 - \int_{0}^{t} v_2(s) ds + U_1(t), \notag \\
&\int_{0}^{\infty} v_1(s) dU_1(s) = 0, \quad U_1(t) \geq 0, \quad  t \geq 0. \label{eq:dynamic}
\end{align}
We refer to $v^{x}(t)$ as the fluid-model corresponding to the fluid-scaled CTMC $\{X(t)\}_{t\geq 0}$. The key idea is that
\begin{align}
 f^*(x) =\int_{0}^{\infty} \big((v_2^{(x)}(s) - \kappa/\sqrt{n}) \vee 0\big) ds \label{eq:ourdfl}
\end{align}
will satisfy the PDE in Lemma~\ref{lem:pde}. Our plan is to a) better understand the behavior of the fluid model and b) to use this knowledge to obtain a closed form representation of \eqref{eq:ourdfl} and bound its derivatives. 
\begin{remark}
Our choice of $f^*(x) =\int_{0}^{\infty} \big((v_2^{x}(s) - \kappa/\sqrt{n}) \vee 0\big) ds$ is a special case of a DFL Lyapunov function. More generally, a DFL Lyapunov function is any function of the form 
\begin{align}
f^{(h)}(x) = \int_{0}^{\infty} h(v^x(s)) ds \label{eq:dflgeneral}
\end{align} 
and we expect it to satisfy 
\begin{align}
&Lf^{(h)}(x) = -h(x), \label{eq:pdefirst}
\end{align} 
where $L$ is the `generator' of the fluid model. 
Section 2 of \cite{Stol2015} proves rigorously that \eqref{eq:pdefirst} is indeed true provided the fluid model satisfies $\frac{d}{dt}	v(t) = F(v(t))$ for some continuous vector field $F(\cdot)$. In our case the vector field is discontinuous because we deal with a linear switching system (more on this in Section~\ref{sec:tau}), and we would need to verify by hand that \eqref{eq:pdefirst} is satisfied for any choice of $h(x)$. 
\end{remark}

 



\subsection{Understanding the fluid model.} \label{sec:tau}
 The following is a heuristic description of the fluid model in \eqref{eq:dynamic}. We refer to $U_1(t)$  as the regulator, because it prevents $v^{x}_1(t)$ from becoming positive. In the absence of this regulator, i.e.\ $U_1(t) \equiv 0$, the system would have been a linear dynamical system 
\begin{align}
\dot v = F(v), \quad \text{ where } \quad F(v) = (-v_1+v_2 - \beta/\sqrt{n},-v_2). \label{eq:fluidlinear}
\end{align}
However, due to the presence of the regulator, for values in the set $\{v_1 = 0,\ v_2 \geq \beta/\sqrt{n}\}$ it is as if  the vector field becomes 
\begin{align}
F(v) = (0,-\beta/\sqrt{n}). \label{eq:fluidlinear2}
\end{align}
Hence, we have a piece-wise linear system, whose dynamics are further illustrated in Figure~\ref{fig:dynamics}.

The fluid model can also be characterized analytically. Suppose the initial condition $x_1< 0$. Then the system behaves according to \eqref{eq:fluidlinear}, meaning that until the vertical axis is hit, i.e. for $t \in [0, \inf_{s\geq 0} \{v_1^x(s) = 0\}]$, its solution is  
\begin{align*}
\begin{pmatrix}
v_1^x(t) \\
v_2^x(t)
\end{pmatrix} = 
\begin{pmatrix}
&-\beta/\sqrt{n} + (x_1 + \beta/\sqrt{n}) e^{-t} + t x_2 e^{-t}  \\
& x_2 e^{-t}  
\end{pmatrix}.
\end{align*}
After the vertical axis is hit, the drift switches to \eqref{eq:fluidlinear2} and $v_2^x(t)$ decreases linearly at a rate $-\beta/\sqrt{n}$ until the point $(0,\beta/\sqrt{n})$ is reached, after which the system behaves according to \eqref{eq:fluidlinear} again.  

\begin{figure}
\hspace{-2.5cm}
\begin{tikzpicture}
 \node (ref) at (0,0){};
\draw[thick] (2,1) -- (8,1)
(7,0)--(7,6);
\node at (6,0.6) { $-\beta/\sqrt{n}$};
\node at (7.6,2) { $\beta/\sqrt{n}$};
\draw[mark options={fill=red}]
      plot[mark=*] coordinates {(6,1)};

\draw[blue,thick,->] (3.2,1) -- (3.6,1);
\draw[blue,thick,->] (4,1) -- (4.5,1);
\draw[blue,thick,->] (3.2,1.5) -- (3.7,1.49);
\draw[blue,thick,->] (4.1,1.485) -- (4.6,1.47);
\draw[blue,thick,->] (5,1.46) -- (5.5,1.4);
\draw[blue,thick,->] (3.7,2) -- (4.2,1.95);
\draw[blue,thick,->] (4.5,1.9) -- (5.2,1.75);
\draw[blue,thick,->] (5.7,1.7) -- (6.2,1.4);
\draw[blue,thick,->] (4.5,2.5) -- (5,2.4);
\draw[blue,thick,->] (5.5,2.3) -- (6,2);

\draw[blue,thick,->] (6.7,1.7) -- (6.6,1.4);
\draw[blue,thick,->] (6.6,1.2) -- (6.4,1.1);

\draw[blue,thick,->] (3.2,1) -- (3.6,1);
\draw[blue,thick,->] (4,1) -- (4.5,1);
\draw[blue,thick,->] (3.2,1.5) -- (3.7,1.49);
\draw[blue,thick,->] (4.1,1.485) -- (4.6,1.47);
\draw[blue,thick,->] (5,1.46) -- (5.5,1.4);
\draw[blue,thick,->] (3.7,2) -- (4.2,1.95);
\draw[blue,thick,->] (4.5,1.9) -- (5.2,1.75);
\draw[blue,thick,->] (5.7,1.7) -- (6.2,1.4);
\draw[blue,thick,->] (4.5,2.5) -- (5,2.4);
\draw[blue,thick,->] (5.5,2.3) -- (6,2);

\draw[blue,thick,->] (6.7,1.7) -- (6.6,1.4);
\draw[blue,thick,->] (6.6,1.2) -- (6.4,1.1);
\draw[dashed] (7,2) .. controls (6,2.7) and (5,3.2) .. (2,3.6); 

 \end{tikzpicture}
 \begin{tikzpicture}
 \node (ref) at (0,0){};
\draw[thick] (2,1) -- (8,1)
(7,0)--(7,6);
\node at (6,0.6) {$-\beta/\sqrt{n}$};
\node at (7.6,2) { $\beta/\sqrt{n}$};
\draw[mark options={fill=red}]
      plot[mark=*] coordinates {(6,1)};

\draw[dashed] (7,2) .. controls (6,2.7) and (5,3.2) .. (2,3.6); 
\draw[blue,thick,->] (6,4.5) -- (6.4,4.35);
\draw[blue,thick,->] (6.6,4.2) -- (7,4);
\draw[red,line width=0.5mm,->] (7,4)-- (7,2);
\draw[blue,thick,->] (7,2) -- (6.6,1.4);
 \end{tikzpicture}
 \caption{Dynamics of the fluid model. Any trajectory starting below the dahsed curve will not hit the vertical axis, and anything starting above the curve will hit the axis and travel down  until reaching the point $(0,\beta/\sqrt{n})$.  } \label{fig:dynamics}
\end{figure}
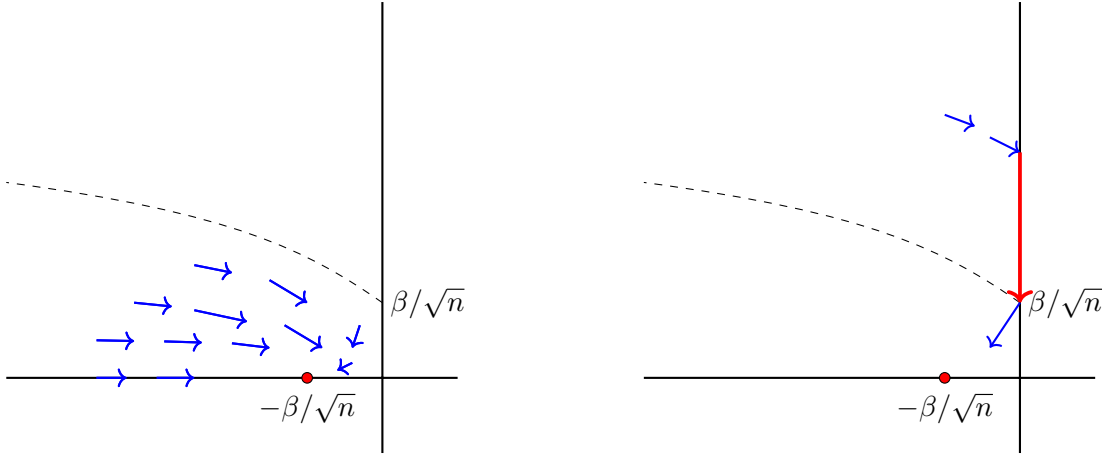

We require two elements to characterize $v^{x}(t)$. The first is the hitting time
\begin{align}
\inf \{t \geq 0 :   v^{x}_1(t) = 0 \}
\end{align}
which is the first hitting time of the vertical axis given initial condition $x$.
 The second is a curve $\Gamma^{(\kappa)} \subset \Omega $. The curve is defined such that for any point $x \in \Gamma^{(\kappa)}$, the fluid path $v^{x}(t)$ first hits the vertical axis at the point $(0,\kappa/\sqrt{n})$. The following two lemmas present rigorous definitions of $\Gamma^{(\kappa)}$ and the hitting time. Lemma~\ref{lem:gamma} is proved in Section~\ref{app:gamma}. 
\begin{lemma} \label{lem:gamma}
Fix $\kappa \geq \beta$ and $x_1 \leq 0$. The nonlinear system 
\begin{align}
&-\beta/\sqrt{n} + (x_1 + \beta/\sqrt{n}) e^{-\eta} + \eta \nu e^{-\eta} = 0, \notag \\
& \nu e^{-\eta} = \kappa/\sqrt{n},\notag \\
&\nu \geq \kappa/\sqrt{n}, \quad \eta \geq 0. \label{eq:nonlin}
\end{align}
has exactly one solution $(\nu^*(x_1), \eta^*(x_1))$. Furthermore, for every $x_1 \leq 0$, let us define the curve 
\begin{align*}
\Gamma^{(\kappa)} = \{ x \in \Omega \ | \ x_2 = \nu^{*}(x_1) \}
\end{align*}
and let 
\begin{align*}
 \gamma^{(\kappa)}(x_1) =  \Big\{ \big(-\beta/\sqrt{n} + (x_1 + \beta/\sqrt{n}) e^{-t} + t \nu^*(x_1) e^{-t}, \nu^*(x_1) e^{-t} \big) \ \Big| \ t \in [0,\eta^*(x_1)] \Big\}.
\end{align*}
Then $\gamma^{(\kappa)}(x_1) \subset \Gamma^{(\kappa)}$ for every $x_1 \leq 0$. 
\end{lemma}
Given $\kappa \geq \beta$ and $x\in \Omega$, let $\Gamma^{(\kappa)}$ and $\nu^{*}(x_1)$ be as in Lemma~\ref{lem:gamma}. Let us adopt the convention of writing
\begin{align}
x > \Gamma^{(\kappa)} \text{ if } x_2 > \nu^{*}(x_1), \label{eq:xgeqgamma}
\end{align} 
and define $x \geq \Gamma^{(\kappa)}$, $x < \Gamma^{(\kappa)}$, and $x   \leq  \Gamma^{(\kappa)}$ similarly. Observe that the sets
\begin{align*}
\{x \in \Omega \ |\ x > \Gamma^{(\kappa)}\}, \quad  \{x \in \Omega \ |\ x < \Gamma^{(\kappa)}\}, \quad \text{ and }  \quad  \{x \in \Omega \ |\ x \in \Gamma^{(\kappa)}\}
\end{align*}
 are disjoint, and that their union equals $\Omega$. Furthermore, 
 \begin{align*}
 \{x \in \Omega \ |\ x \geq \Gamma^{(\kappa)}\} \cap \{x \in \Omega \ |\ x \leq \Gamma^{(\kappa)}\} =  \{x \in \Omega \ |\ x \in \Gamma^{(\kappa)}\}.
\end{align*} 
The next lemma characterizes the first hitting time of the vertical axis given initial condition $x$, and shows that this hitting time is differentiable in $x$. It is proved in Section~\ref{app:tau}.
\begin{lemma}  \label{lem:tau}
Fix $\kappa \geq \beta$ and $x \in  (-\infty,0]\times [\kappa/\sqrt{n},\infty)$. Provided it exists, define $\tau(x)$ to be the smallest solution to 
\begin{align*}
\beta/\sqrt{n} - (x_1 + \beta/\sqrt{n}) e^{-\eta} - \eta x_2 e^{-\eta} = 0, \quad \eta \geq 0,
\end{align*}
and define $\tau(x) = \infty$ if no solution exists. It follows immediately that 
\begin{align}
\tau(0,x_2) = 0, \quad x_2 \geq 0. \label{eq:tauboundary}
\end{align}
Furthermore, let $\Gamma^{(\kappa)}$ be as in Lemma~\ref{lem:gamma}.   
\begin{enumerate}
\item  If $x > \Gamma^{(\kappa)}$, then $\tau(x) < \infty$ and 
\begin{align}
x_2 e^{-\tau(x)} > \kappa/\sqrt{n}, \label{eq:gammaprop}
\end{align} 
and if $x \in  \Gamma^{(\kappa)}$, then $\tau(x) < \infty$ and $x_2 e^{-\tau(x)} = \kappa/\sqrt{n}$.  
\item If $\kappa > \beta$, then the function $\tau(x)$ is differentiable at all points $x \geq \Gamma^{(\kappa)}$ with
\begin{align}
\tau_1(x) = - \frac{ e^{-\tau(x)}}{x_2 e^{-\tau(x)} - \beta/\sqrt{n}} \leq 0, \quad \tau_2(x) = \tau_1(x) \tau(x) \leq 0, \quad x \geq \Gamma^{(\kappa)}, \label{eq:tauder}
\end{align}
where $\tau_1(x)$ is understood to be the left derivative when $x_1 = 0$. 
\item 
 For any $\kappa_1,\kappa_2$ with $\beta < \kappa_1 < \kappa_2$, 
\begin{align}
x \geq \Gamma^{(\kappa_2)} \text{ implies } x > \Gamma^{(\kappa_1)}, \label{eq:gammaabove}
\end{align}
i.e. the curve $\Gamma^{(\kappa_2)}$ lies strictly above $\Gamma^{(\kappa_1)}$.
\end{enumerate}
\end{lemma}
Armed with Lemmas~\ref{lem:gamma} and \ref{lem:tau}, we 
 are now in a position to present the function $f^*(x)$ that will satisfy the PDE in Lemma~\ref{lem:pde}. 

\subsection{Constructing $f^*(x)$} \label{sec:candidate}
Fix $\kappa > \beta$ and partition the set $\Omega$ into three subdomains 
\begin{align*}
 \{x_2 \in [0,\kappa/\sqrt{n}]\}, \quad  \{x \leq \Gamma^{(\kappa)},\ x_2 \geq \kappa/\sqrt{n} \}, \quad \text{ and } \quad \{x \geq \Gamma^{(\kappa)} \},
\end{align*}
where $\Gamma^{(\kappa)}$ is as in Lemma~\ref{lem:gamma}. From \eqref{eq:gammaprop} we know that $x \geq \Gamma^{(\kappa)}$ implies $x_2 \geq \kappa/\sqrt{n}$, and therefore any point in $\Omega$ must indeed lie in one of the three subdomains. The following is an informal discussion of the intuition behind the form of $f^*(x)$, which is given in \eqref{eq:candidate} below. 

We already said that we will choose
\begin{align*}
f^*(x) =  \int_{0}^{\infty} \big((v_2^{x}(s) - \kappa/\sqrt{n}) \vee 0\big) ds,
\end{align*}
and so we aim to understand the integral on the right hand side.
Recall that the fluid model is a piece-wise linear model satisfies \eqref{eq:fluidlinear} off the vertical boundary $\{x_1=0\}$, and \eqref{eq:fluidlinear2} on the vertical boundary. The simplest case to work with is if $x_2 \in [0,\kappa/\sqrt{n}]$.  In both \eqref{eq:fluidlinear} and \eqref{eq:fluidlinear2}, the $v_2$ component has a negative drift, meaning $v_2^{x}(s) \leq v_2^{x}(0)$ for all $s \geq 0$. Therefore, we let 
\begin{align*}
 f^*(x)= \int_{0}^{\infty} \big((v_2^{x}(s) - \kappa/\sqrt{n}) \vee 0\big) ds = 0, \quad \text{ if }  x_2 \in [0,\kappa/\sqrt{n}].
\end{align*}

Now suppose $x_2 \geq \kappa/\sqrt{n}$ and $x \leq \Gamma^{(\kappa)}$. This means  that the fluid model's point of contact with the vertical axis is upper bounded by $\kappa/\sqrt{n}$. Therefore, the fluid model simply behaves like the linear system in \eqref{eq:fluidlinear} all the way until $v_2^x(t) = \kappa/\sqrt{n}$, after which time  $\big((v_2^{x}(s) - \kappa/\sqrt{n}) \vee 0\big)$ becomes zero. This tells us that
\begin{align*}
\int_{0}^{\infty} \big((v_2^{x}(s) - \kappa/\sqrt{n}) \vee 0\big) ds =&\ \int_{0}^{\inf_{t \geq 0} \{v_2^{x}(t) = \kappa/\sqrt{n}\}}  \big((v_2^{x}(s) - \kappa/\sqrt{n}) \vee 0\big)  ds \\
=&\ \int_{0}^{\inf_{t \geq 0} \{v_2^{x}(t) = \kappa/\sqrt{n}\}}  (x_2 e^{-s} - \kappa/\sqrt{n})  ds \\
=&\ \int_{0}^{\log(x_2\sqrt{n}/\kappa)}  (x_2 e^{-s} - \kappa/\sqrt{n})  ds \\
=&\  x_2 (1 - \kappa/x_2 \sqrt{n}) - \kappa/\sqrt{n}\log (x_2 \sqrt{n}/\kappa),
\end{align*}
where in the second and third equalities we used the fact that $\dot v_2^x(t) = -v_2^x(t)$ or $v_2^x(t) = x_2 e^{-t}$, and  that the time until $v_2^x(t)=x_2 e^{-t}$ hits $\kappa/\sqrt{n}$ is $\log(x_2\sqrt{n}/\kappa)$. Therefore, for $x_2 \geq \kappa/\sqrt{n}$ and $x \leq \Gamma^{(\kappa)}$ we set 
\begin{align*}
f^*(x) = x_2   - \kappa/ \sqrt{n}  - \kappa/\sqrt{n}\log (x_2 \sqrt{n}/\kappa).
\end{align*}

Lastly, if $x \geq \Gamma^{(\kappa)}$, then the fluid model behaves like \eqref{eq:fluidlinear} from time $0$ until $\tau(x)$, at which point it hits the vertical axis (above $(0,\kappa/\sqrt{n})$). Once the fluid model hits the vertical axis, the $v_2(t)$ component decreases linearly at a rate $-\beta/\sqrt{n}$ until the point $(0,\kappa/\sqrt{n})$ is hit; afterwards, $\big((v_2^{x}(s) - \kappa/\sqrt{n}) \vee 0\big)=0$. Therefore, $\int_{0}^{\infty} \big((v_2^{x}(s) - \kappa/\sqrt{n}) \vee 0\big) ds$ is split into two parts:
\begin{align*}
\int_{0}^{\tau(x)}  \big((v_2^{x}(s) - \kappa/\sqrt{n}) \vee 0\big)  ds + \int_{\tau(x)}^{\infty}  \big((v_2^{x}(s) - \kappa/\sqrt{n}) \vee 0\big)  ds.
\end{align*}
The first term (before the vertical axis is hit) equals
\begin{align*}
& \int_{0}^{\tau(x)}  (x_2 e^{-s} - \kappa/\sqrt{n})   ds =  x_2 (1-e^{-\tau(x)}) - \tau(x) \kappa/\sqrt{n}
\end{align*}
To evaluate the second term, we observe that once the vertical axis is hit at time $\tau(x)$, the $v_2^*(t)$ component decreases linearly at a rate of $\beta/\sqrt{n}$ until it hits the level $\kappa/\sqrt{n}$ (which occurs after $( v_2^x(\tau(x)) - \kappa/\sqrt{n})/\beta/\sqrt{n}$ time units). Therefore, 
\begin{align*}
 \int_{\tau(x)}^{\infty}  \big((v_2^{x}(s) - \kappa/\sqrt{n}) \vee 0\big)  ds =&\  \int_{\tau(x)}^{\tau(x) + ( v_2^x(\tau(x)) - \kappa/\sqrt{n})/\beta/\sqrt{n}}  (v_2^{x}(s) - \kappa/\sqrt{n})    ds \\
=&\  \int_{0}^{ ( v_2^x(\tau(x)) - \kappa/\sqrt{n})/\beta/\sqrt{n}}  (v_2^{x}(\tau(x) + s) - \kappa/\sqrt{n})    ds \\
=&\   \int_{0}^{ ( v_2^x(\tau(x)) - \kappa/\sqrt{n})/\beta/\sqrt{n}}   (x_2 e^{-\tau(x)}- \kappa/\sqrt{n}  - s\beta/\sqrt{n} )    ds \\
=&\   \frac{1}{2}\frac{\sqrt{n}}{\beta}\big(x_2 e^{-\tau(x)} - \kappa/\sqrt{n} \big)^2.
\end{align*}
We conclude our above discussion by defining 
\begin{align}
f^*(x) = 
\begin{cases}
&0, \hfill x_2 \in [0,\kappa/\sqrt{n}],\\
&x_2- \frac{\kappa}{ \sqrt{n}}  - \frac{\kappa}{\sqrt{n}} \log (\sqrt{n}x_2/\kappa),  \hfill  x \leq \Gamma^{(\kappa)} \text{ and } x_2 \geq \kappa/\sqrt{n}, \\
&x_2(1 - e^{-\tau(x)}) - \frac{\kappa}{\sqrt{n}} \tau(x) +   \frac{1}{2} \frac{\sqrt{n}}{\beta} (x_2 e^{-\tau(x)} - \kappa/\sqrt{n})^2,  \quad  x \geq \Gamma^{(\kappa)}.
\end{cases} \label{eq:candidate} 
\end{align}
Note that the above discussion is informal  in the sense that we did not actually prove anything rigorous about the fluid model $v^x(t)$. Instead, we simply came up with a candidate PDE solution $f^*(x)$ based on an intuitive grasp of the fluid model. Nevertheless, the following Lemma confirms our intuition and shows that $f^*(x)$ does indeed solve the PDE; it is proved in Section~\ref{app:dertech}.

\begin{lemma} \label{lem:derivs}
The function $f^*(x)$ in \eqref{eq:candidate} is well-defined, has  $f_1^*(\cdot,x_2), f_{2}^*(x_1,\cdot)$ absolutely continuous for  all $x \in \Omega$, satisfies the PDE \eqref{eq:poisson1}--\eqref{eq:poisson2}, and satisfies the derivative bounds in \eqref{eq:dpos}--\eqref{eq:d22}. 
\end{lemma}
 Lemma~\ref{lem:derivs} was the final piece in the proof of Theorem~\ref{thm:main}.

\section{The diffusion limit: exponential ergodicity.} \label{sec:difflimit}
Theorem~\ref{thm:transient} proves that $\{\sqrt{n}(X_1(t),  X_2(t))\}_{t\geq 0}$ converges to a diffusion limit. Convergence was established only over finite time intervals, but convergence of steady-state distributions was not justified. In fact, it has not been shown that the process in \eqref{eq:diffusion} is even positive recurrent. We show that not only is this process positive recurrent (Theorem~\ref{thm:ergodic}), but it is also exponentially ergodic. The proof involves a very similar approach to that of Theorem~\ref{thm:main}. Namely, our proof will again revolve around  comparing the diffusion generator to its fluid model counterpart.

Recall that the diffusion limit in Theorem~\ref{thm:transient} is 
\begin{align*}
&Y_1(t) = Y_1(0) + \sqrt{2} W(t) - \beta t + \int_{0}^{t} (-Y_1(s) + Y_2(s)) ds - U(t), \notag \\
&Y_2(t) = Y_2(0) + U(t)  - \int_{0}^{t} Y_2(s) ds,
\end{align*}
where $U(t)$ is the regulator and satisfies $\int_{0}^{\infty} 1(Y_1(t)<0) d U(t) = 0$. To discuss geometric ergodicity, we introduce the extended generator of this diffusion process. Denote by $D(G_Y)$ the set of all functions $f:\Omega \to \R$ for which there exists a measurable function $g: \Omega \to \R$ such that, for each $x \in \Omega$, $t \geq 0$, 
\begin{align}
 \E_{x} f(Y_1(t),Y_2(t)) - \E_x f(Y_1(0),Y_2(0)) =&\ \E_x \int_{0}^{t} g(Y_1(s),Y_2(s)) ds, \label{eq:extendedgen}\\
 \int_{0}^{t}\E_x \abs{g(Y_1(s),Y_2(s))} ds < \infty. \notag
\end{align}
We write $G_Y f = g$ and call $G_Y$ the extended generator of $\{Y(t)\}$. An application of Ito's lemma (see \cite[Theorem 2]{HarrReim1981} for an example of Ito's lemma in the presence of regulators)  states that for any function $f(x) \in C^2(\Omega)$,
\begin{align*}
& \E_{x} f(Y_1(t),Y_2(t)) - \E_x f(Y_1(0),Y_2(0)) \\
=&\ \E_x \int_{0}^{t}\Big( (-Y_1(s) + Y_2(s) - \beta)  f_1(Y_1(s),Y_2(s)) - Y_2(s) f_2(Y_1(s),Y_2(s)) + f_{11}(Y_1(s),Y_2(s)) \Big)ds \\
&+ \E_x\int_{0}^{t} \big(-f_{1}(0,Y_2(s)) + f_2(0,Y_2(s))\big) dU(s).
\end{align*}
Comparing the above expansion to \eqref{eq:extendedgen}, we see that for functions $f(x)$ with $f_1(0,x_2) = f_2(0,x_2)$,
\begin{align*}
G_Y f(x) = (-x_1 + x_2- \beta ) f_1(x) - x_2 f_2(x) + f_{11}(x), \quad x \in \Omega.
\end{align*}
The following theorem proves the existence of a function satsfying a Foster-Lyapunov condition that is needed for exponential ergodicity.
\begin{theorem} \label{thm:expergodic}
Fix $\beta> 0$. There exist positive constants $c$ and $d$, a compact set $K$, and a function $V : \Omega \to [1,\infty)$ with $V(x) \to \infty$ as $\abs{x } \to \infty$ and $V_1(0,x_2)=V_2(0,x_2)$ such that
\begin{align}
G_Y V(x) \leq&\ - c V(x) + d 1(x \in K), \label{eq:expgen}
\end{align}
The function $V(x)$, $c$, $d$, and $K$ all depend on $\beta$.
\end{theorem}
The proof of the theorem is given in Section~\ref{sec:proofergod}. A consequence of Theorem~\ref{thm:expergodic} is exponential ergodicity of the diffusion in \eqref{eq:diffusion}: the following corollary is an immediate consequence of \eqref{eq:expgen} and \cite[Theorem 5.2]{DownMeynTwee1995}.
\begin{corollary}
The diffusion process $\{(Y_1(t),Y_2(t))\}_{t \geq 0}$ defined in \eqref{eq:diffusion} is positive recurrent. Furthermore, let  $Y = (Y_1,Y_2)$ be the vector having its stationary distribution, and let $V(x)$ be the function from Theorem~\ref{thm:expergodic}. There exist constants $b < 1$ and $B < \infty$ such that 
\begin{align*}
\sup_{\abs{f}\leq V} \abs{ \E_{x} f(Y(t)) - \E f(Y)} \leq B V(x) b^{t}
\end{align*}
\end{corollary}

\subsection{Proving Theorem~\ref{thm:expergodic}.} \label{sec:proofergod}
The proof of Theorem~\ref{thm:expergodic} follows a similar line of reasoning as the proof of Theorem~\ref{thm:main}. Namely, we view the diffusion generator as 
\begin{align*}
G_Y f(x) = L f(x) + \text{error},
\end{align*}
and we choose a function such that $Lf(x)$ is well behaved. The error term will contain derivatives of $f(x)$, and so we want those to be controlled as well. Let us examine how $G_Y$ acts on Lyapunov functions of the form $V(x) = e^{f(x/\sqrt{n})}$ (assuming for now that $V_1(0,x_2) = V_2(0,x_2)$):
\begin{align} 
 G_Y V(x) =&\ (-x_1 + x_2 - \beta) V_1(x) - x_2 V_2(x) + V_{11}(x) \notag \\
=&\ (-x_1/\sqrt{n} + x_2/\sqrt{n} - \beta/\sqrt{n}) (f_1 (x/\sqrt{n}))\alpha V(x) \notag \\
&- \frac{x_2}{\sqrt{n}}(f_2 (x/\sqrt{n}))\alpha V(x)  + V_{11}(x)\notag \\
=&\ (L f (x/\sqrt{n}) )\alpha V(x) + \frac{1}{n}\big(\alpha f_{11} (x/\sqrt{n}) + \alpha^2 (f_1 (x/\sqrt{n}))^2 \big)V(x).\label{eq:diffusiongen}
\end{align}
Therefore, to satisfy a condition like \eqref{eq:expgen}, it suffices to choose a function $f(x)$ such that $L f(x/\sqrt{n}) \leq -c$ and both  $f_{1}(x/\sqrt{n})$ and $f_{11}(x/\sqrt{n})$ are sufficiently under control. One candidate is to set $V(x)$ equal to 
\begin{align}
\exp \Big(\int_{0}^{\infty} 1\big(v^{x/\sqrt{n}}(t) \not \in [-\kappa/\sqrt{n},0]\times[0,\kappa/\sqrt{n}]\big) \Big), \label{eq:hittime}
\end{align}
i.e. the exponential of the fluid hitting time to the set $[-\kappa/\sqrt{n},0] \times [0,\kappa/\sqrt{n}]$. The integral in the exponent is a DFL Lyapunov function like in \eqref{eq:dflgeneral}, and so we hope that 
\begin{align*}
L \int_{0}^{\infty} 1\big(v^{x/\sqrt{n}}(t) \not \in [-\kappa/\sqrt{n},0]\times[0,\kappa/\sqrt{n}]\big)  = -1\big({x/\sqrt{n}} \not \in [-\kappa/\sqrt{n},0]\times[0,\kappa/\sqrt{n}]\big).
\end{align*}
However, we cannot use \eqref{eq:hittime} directly because $G_Y$ acts on $C^2(\Omega)$ functions, and \eqref{eq:hittime} does not have the required regularity; the indicator inside the integral is a discontinuous function. Instead, we will use a smoothed relative of \eqref{eq:hittime}. Let us define a smoothed indicator. For any $\ell  < u$,  let 
\begin{align}
\phi^{(\ell,u)}(x) = 
\begin{cases}
0, \quad &x \leq \ell,\\
(x-\ell)^2\Big( \frac{-(x-\ell)}{ ((u+\ell)/2-\ell )^2(u-\ell)} + \frac{2}{ ((u+\ell)/2-\ell)(u-\ell)}\Big), \quad &x \in [\ell,(u+\ell)/2],\\
1 - (x-u)^2\Big( \frac{(x-u)}{ ((u+\ell)/2-u )^2(u-\ell)} - \frac{2}{ ((u+\ell)/2-u )(u-\ell)}\Big), \quad &x \in [(u+\ell)/2,u],\\
1, \quad &x \geq u.
\end{cases} \label{eq:phi}
\end{align}
It is straightforward to check that $\phi^{(\ell,u)}(x)$ has an absolutely continuous first derivative, and that 
\begin{align}
&\ (\phi^{(\ell,u)})'(\ell) = (\phi^{(\ell,u)})'(u) = 0 \label{eq:phiderzero}\\
&\ \abs{(\phi^{(\ell,u)})'(x)} \leq \frac{4}{u-\ell}, \quad \text{ and } \quad \abs{(\phi^{(\ell,u)})''(x)} \leq \frac{12}{(u-\ell)^2} . \label{eq:phider}
\end{align}
Fix $\kappa_2>\kappa_1>\beta$ and $\alpha \in (0,1)$. The Lyapunov function we will use to prove Theorem~\ref{thm:expergodic} is    
\begin{align}
V^{(\kappa_1,\kappa_2)}(x) = \exp\big(\alpha (f^{(1)}(x/\sqrt{n}) + f^{(2)}(x/\sqrt{n}))\big)
\end{align}
where $f^{(1)}(x)$ and $f^{(2)}(x)$ will be DFL Lyapunov functions constructed to satisfy 
\begin{align}
L f^{(1)}(x) =&\ -\phi^{(\kappa_1/\sqrt{n} ,\kappa_2/\sqrt{n})} (-x_1), \quad x \in \Omega, \notag\\
f_1^{(1)}(0,x_2) =&\ f_2^{(1)}(0,x_2), \quad   x_2 \geq 0, \label{eq:pde1}
\end{align} 
and 
\begin{align}
L f^{(2)}(x) =&\ -\phi^{(\kappa_1/\sqrt{n} ,\kappa_2/\sqrt{n})} (x_2), \quad x \in \Omega, \notag\\
f_1^{(2)}(0,x_2) =&\ f_2^{(2)}(0,x_2), \quad   x_2 \geq 0. \label{eq:pde2}
\end{align} 
Let us omit the superscript from $V(x)$ for convenience and set 
\begin{align*}
f^{(\Sigma)} (x) = f^{(1)}(x) + f^{(2)}(x).
\end{align*} 
Observe that $V_1(0,x_2)=V_2(0,x_2)$ for all $x_2\geq 0$ by \eqref{eq:pde1}--\eqref{eq:pde2}. 
Since $L f^{(\Sigma)} (x/\sqrt{n}) = Lf^{(1)} (x/\sqrt{n}) + Lf^{(2)} (x/\sqrt{n})$, 
\begin{align}
& G_Y V(x) \notag  \\
=&\ (L f^{(\Sigma)} (x/\sqrt{n}) )\alpha V(x) + \frac{1}{n}\big(\alpha f_{11}^{(\Sigma)} (x/\sqrt{n}) + \alpha^2 (f_1^{(\Sigma)} (x/\sqrt{n}))^2 \big)V(x) \notag \\
=&\ (-\phi^{(\kappa_1/\sqrt{n} ,\kappa_2/\sqrt{n})} (-x_1/\sqrt{n}) -\phi^{(\kappa_1/\sqrt{n} ,\kappa_2/\sqrt{n})} (x_2/\sqrt{n}) )\alpha V(x)  \notag \\
&+ \frac{1}{n}\big(\alpha f_{11}^{(\Sigma)} (x/\sqrt{n}) + \alpha^2 (f_1^{(\Sigma)} (x/\sqrt{n}))^2 \big)V(x) \notag  \\
&\leq -\alpha V(x) 1(x \not \in [-\kappa_2,0]\times [0,\kappa_2])+ \frac{1}{n}\big(\alpha f_{11}^{(\Sigma)} (x/\sqrt{n}) + \alpha^2 (f_1^{(\Sigma)} (x/\sqrt{n}))^2 \big)V(x) \notag \\
=&\  \alpha\Big(-1 +  \frac{1}{n}\big(f_{11}^{(\Sigma)} (x/\sqrt{n}) + \alpha (f_1^{(\Sigma)} (x/\sqrt{n}))^2 \big)\Big)V(x) \label{eq:GY1}\\
&+ \frac{1}{n}\big(\alpha f_{11}^{(\Sigma)} (x/\sqrt{n}) + \alpha^2 (f_1^{(\Sigma)} (x/\sqrt{n}))^2 \big)V(x) 1(x  \in [-\kappa_2,0]\times [0,\kappa_2]), \label{eq:GY2}
\end{align}
where in the first inequality we used the fact that $\phi^{(\kappa_1/\sqrt{n} ,\kappa_2/\sqrt{n})} (x/\sqrt{n}) = 1 $ for $x \geq \kappa_2$. Let us compare \eqref{eq:GY1}--\eqref{eq:GY2} to our desired result in \eqref{eq:expgen} to see that we need bounds on  $f_1^{(\Sigma)} (x)$, $f_{11}^{(\Sigma)}(x)$, and $V(x)1(x  \in [-\kappa_2,0]\times [0,\kappa_2])$. The following lemma presents all the properties of $f^{(1)}(x)$ and $f^{(2)}(x)$ and their derivatives that we will need to prove Theorem~\ref{thm:expergodic}. It is proved in Section~\ref{app:expergodderivs}. 
\begin{lemma} \label{lem:expergodderivs}
Fix $\kappa_1< \kappa_2$  such that $\kappa_1 > \beta$. Then both PDE's \eqref{eq:pde1} and \eqref{eq:pde2} have solutions $f^{(1)}(x)$ and $f^{(2)}(x)$, respectively. The solutions belong to $C^2(\Omega)$, and for every $\epsilon > 0$,
\begin{align}
&f^{(1)}(x) \leq \log 2, \quad &x \in [-\kappa_2/\sqrt{n},0]\times [0,\kappa_2/\sqrt{n}],\label{eq:expfsum1}\\
&  f^{(2)}(x) \leq   \log 2 + \frac{\epsilon}{\beta}, \quad &x \in [-\kappa_2/\sqrt{n},0]\times [0,\kappa_2/\sqrt{n}], \label{eq:expfsum2}
\end{align}
 and for all $x \in \Omega$,
\begin{align}
&\big| f_{1}^{(1)}(x) \big| \leq \frac{4\sqrt{n}}{\epsilon}\log 2, \quad  \big|f_{11}^{(1)}(x)\big| \leq \frac{12n}{\epsilon^2}\log 2, \label{eq:expf1} \\
&\big| f_{1}^{(2)}(x) \big| \leq \frac{\sqrt{n}}{\beta}, \quad \big|f_{11}^{(2)}(x)\big| \leq \frac{n}{\beta\epsilon}\Big(1  + 4\frac{\beta+\epsilon}{\epsilon}   \Big). \label{eq:expf2}
\end{align}
\end{lemma} 
With these derivative bounds, we are ready to prove Theorem~\ref{thm:expergodic}.
\begin{proof}[Proof of Theorem~\ref{thm:expergodic}]
Fix $\kappa_1 < \kappa_2$ with $\kappa_1 > \beta$ and $\alpha > 0$, let $f^{(1)}(x)$ and $f^{(2)}(x)$ be as in Lemma~\ref{lem:expergodderivs}, and let $V^{(\kappa_1,\kappa_2)}(x) = e^{\alpha (f^{(1)}(x/\sqrt{n}) + f^{(2)}(x/\sqrt{n}))}$. Our goal is to find positive constants $c,d$ such that \eqref{eq:expgen} is satisfied. It follows from \eqref{eq:GY1}--\eqref{eq:GY2} that
\begin{align*}
G_Y V(x) \leq&\ \alpha\Big(-1 +  \frac{1}{n}\big(f_{11}^{(\Sigma)} (x/\sqrt{n}) + \alpha (f_1^{(\Sigma)} (x/\sqrt{n}))^2 \big)\Big)V(x) \\
&+ \frac{1}{n}\big(\alpha f_{11}^{(\Sigma)} (x/\sqrt{n}) + \alpha^2 (f_1^{(\Sigma)} (x/\sqrt{n}))^2 \big)V(x) 1(x  \in [-\kappa_2,0]\times [0,\kappa_2])
\end{align*}
By \eqref{eq:expf1}--\eqref{eq:expf2}, 
\begin{align*}
\frac{1}{n}\big|f_{11}^{(\Sigma)} (x/\sqrt{n}) + \alpha (f_1^{(\Sigma)} (x/\sqrt{n}))^2 \big| =& \ \frac{1}{n}\big|f_{11}^{(1)}(x/\sqrt{n})+f_{11}^{(2)} (x/\sqrt{n}) + \alpha (f_1^{(1)} (x/\sqrt{n}) + f_1^{(2)} (x/\sqrt{n}))^2 \big|\\
\leq&\ \frac{1}{n}\Big(\frac{12n}{\epsilon^2}\log 2+\frac{n}{\beta\epsilon}\Big(1  + 4\frac{\beta+\epsilon}{\epsilon}   \Big) + \alpha \big(\frac{4\sqrt{n}}{\epsilon}\log 2 + \frac{\sqrt{n}}{\beta}\big)^2 \Big).
\end{align*}
Note that the right hand side above is independent of $n$. Furthermore, by choosing $\epsilon$ large enough and $\alpha$ small enough, the term on the right hand side can be made to be less than one, meaning there exists a $c > 0$ such that
\begin{align*}
G_Y V(x) \leq&\ -cV(x)+ \frac{1}{n}\big(\alpha f_{11}^{(\Sigma)} (x/\sqrt{n}) + \alpha^2 (f_1^{(\Sigma)} (x/\sqrt{n}))^2 \big)V(x) 1(x  \in [-\kappa_2,0]\times [0,\kappa_2])
\end{align*}
To bound the second term on the right hand side we use \eqref{eq:expfsum1}--\eqref{eq:expfsum2} and the fact that $V(x) = e^{\alpha (f^{(1)}(x/\sqrt{n}) + f^{(2)}(x/\sqrt{n}))}$ tell us that 
\begin{align*}
V(x) 1(x  \in [-\kappa_2,0]\times [0,\kappa_2]) \leq \exp\Big(\log 2 + \log 2 + \frac{\epsilon}{\beta} \Big)1(x  \in [-\kappa_2,0]\times [0,\kappa_2])
\end{align*}
Note that the upper bound does not depend on $n$. Therefore,
\begin{align*}
G_Y V(x) \leq&\ -cV(x) + d 1(x \in K),
\end{align*}
where $K = [-\kappa_2,0]\times [0,\kappa_2]$, and $c,d$ are positive constants that depend only on $\beta$ and no other system parameters such as $\lambda$ or $n$. 
\end{proof}
\begin{remark}
In the proof of Theorem~\ref{thm:ergodic} we compare the generator of the diffusion process $G_Y$ to $L$, which can be thought of as the generator of the associated fluid model. One may wonder why we do not use a similar argument to compare $L$ to $G_X$, and prove that the CTMC is also exponentially ergodic. The answer is that the CTMC is infinite dimensional, while the operator $L$ acts on functions of only two variables. As a result, comparing $G_X$ to $L$ leads to excess error terms that $L$ does not account for, e.g.\ $q_3$ in \eqref{eq:diff2}. Although we were able to get around this issue in the proof of Theorem~\ref{thm:main} by taking expected values, the same trick will not work now because  \eqref{eq:expgen} has to hold for every state. To prove exponential ergodicity, one needs to replace the operator $L$ and the PDE \eqref{eq:poisson1} by infinite-dimensional counterparts corresponding to the infinite-dimensional fluid model of $\{(X_1(t),X_2(t),X_3(t),\ldots) \}_{t\geq 0}$. This is left as an open problem to the interested reader, as Theorem~\ref{thm:ergodic} is sufficient for the purposes of illustrating the proof technique.
\end{remark}

\section{Conclusion}
This paper contains a steady-state analysis of the JSQ model in the Halfin-Whitt regime, using the generator expansion/DFL Lyapunov function methodology to prove tightness of the customer count process. The proof procedure is to 1) write down the CTMC generator 2) perform Taylor expansion on it to extract a fluid model generator and 3) set up a PDE related to the fluid model and bound the derivatives of the solution to said PDE. The bottleneck of this methodology are the derivative bounds of the DFL Lyapunov function; this can only be done if the fluid model is relatively well understood. In addition to proving tightness we saw in Section~\ref{sec:proofergod}, that exponentials of DFL Lyapunov functions can be used to prove exponential ergodicity of a process. 

One important open problem that this paper did not address is the following. When DFL Lyapunov functions were discussed in \cite{Stol2015}, the author considered fluid models with continuous vector fields (i.e.\  $\frac{d}{dt}	v(t) = F(v(t))$ where $F(\cdot)$ is continuous). In that setting, \cite{Stol2015} showed by a simple argument that 
\begin{align}
L \int_{0}^{\infty} h(v^x(s)) ds = - h(x), \label{eq:keyproperty}
\end{align}
where $L$ is the `generator' of the fluid model. Our JSQ model does not satisfy the continuity condition in \cite{Stol2015}. Furthermore,  our PDE has a reflecting boundary condition
\begin{align}
f_1(0,x_2) = f_2(0,x_2), \label{eq:bdrycond}
\end{align} 
which appears due to the presence of the regulator in the fluid model. In this paper we must verify in a brute force manner that our DFL Lyapunov function satisfies both \eqref{eq:keyproperty} and \eqref{eq:bdrycond}. It would be very useful to prove that $\int_{0}^{\infty} h(v^x(s))$ automatically satisfies the aforementioned properties even in the presence of a discontinuous vector field and regulators in the fluid model.

%
%
%

\section*{Acknowledgments.}

This work was inspired by a talk given by David Gamarnik at Northwestern University's Kellogg School of Business in October 2017.

\appendix

\section{Miscellaneous proofs.}
This appendix contains proofs to a few miscellaneous lemmas used in the paper.
\subsection{Lemma~\ref{lem:gz}.} \label{app:gzlemma}
\begin{proof}[Proof of Lemma~\ref{lem:gz}] 
A sufficient condition to ensure that
\begin{align*}
\E \big[ G_{Q} f(Q) \big] = 0
\end{align*}
is given by \cite[Proposition 1.1]{Hend1997} (alternatively, see \cite[Proposition 3]{GlynZeev2008}). Namely, we require that 
\begin{align}
\E \Big[\big| G_{Q} (Q,Q) f(Q)\big| \Big] < \infty, \label{eq:gzcond}
\end{align}
where $G_{Q} (q,q)$ is the diagonal entry of the generator matrix $G_{Q}$ corresponding to state $q \in S$. It is not hard to check that in the JSQ system, $\abs{G_{Q} (q,q)} < n\lambda + n$ for all states $q \in S$. Our assumption that $\E |f(Q)| < \infty$, is enough to ensure \eqref{eq:gzcond} is satisfied.
\end{proof}

\subsection{Lemma~\ref{lem:q1}} \label{app:q1}
\begin{proof}[Proof of Lemma~\ref{lem:q1}]
Fix $M > 0$ and let $f(q) = \min \big( M, \sum_{i=1}^{\infty} q_i\big)$. Then 
\begin{align*}
G_{Q} f(q) = n\lambda 1\big(\sum_{i=1}^{\infty} q_i < M\big)   - q_1 1\big( \sum_{i=1}^{\infty} q_i \leq M\big).
\end{align*}
Using \eqref{eq:bar}, 
\begin{align*}
n\lambda\Prob\big(T < M\big) = \E \Big(Q_1 1\big( T \leq M\big)\Big),
\end{align*}
where $T = \sum_{i=1}^{\infty} Q_i$ is the total customer count. Although the infinite series in the definition of $T$ may seem worrying at first, stability of the JSQ model in fact implies that $T < \infty$ almost surely. To see why this is true, observe that an alternative way to describe the JSQ model is via the CTMC $\{(S_1(t), \ldots, S_n(t))\}_{t \geq 0}$, where $S_i(t)$ be the number of customers assigned to server $i$ at time $t$; we can view $Q(t)$ as a deterministic function of $(S_1(t), \ldots, S_n(t))$. This new CTMC is also positive recurrent, but now the total number of customers in the system at time $t$ is the finite sum $\sum_{i=1}^{n} S_i(t)$. Therefore, $T < \infty$ almost surely, and we can take $M \to \infty$ and apply the monotone convergence theorem to conclude that 
\begin{align*}
\E Q_1 = n\lambda.
\end{align*}
Repeating the argument above with $f(q) = \min \big( M, \sum_{j=i}^{\infty} q_j\big)$ gives us
\begin{align*}
n\lambda \Prob(Q_{1}= \ldots = Q_{i-1}=n) =  \E Q_i.
\end{align*}

\end{proof}

\subsection{Lemma~\ref{lem:gentaylor}.} \label{app:taylor}
\begin{proof}[Proof of Lemma~\ref{lem:gentaylor}]
The CTMC generator satisfies
\begin{align}
G_X A f(q) =&\ n\lambda 1( q_1< n) \big( f(x_1 + 1/n, x_2) - f(x_1,x_2) \big) \notag  \\
& +n\lambda 1( q_1 = n, q_2< n) \big( f(x_1 , x_2+ 1/n) - f(x_1,x_2) \big) \notag \\
& +(q_1 - q_2) \big( f(x_1 -1/n, x_2) - f(x_1,x_2) \big) \notag \\
& +(q_2 - q_3) \big( f(x_1 , x_2-1/n) - f(x_1,x_2) \big). \label{eq:generexplicit}
\end{align}
It is straightforward to verify that
\begin{align}
f(x + e^{(1)}/n) - f(x) =&\ \frac{1}{n} f_1(x) + \int_{x_1}^{x_1+1/n} (x_1 + 1/n - u) f_{11}(u,x_2) du, \notag \\
f(x - e^{(1)}/n) - f(x) =&\ -\frac{1}{n} f_1(x) + \int_{x_1-1/n}^{x_1} (u - (x_1 - 1/n)) f_{11}(u) du, \label{eq:abscont}
\end{align}
and that a similar expansion holds for $f(x + e^{(2)}/n) \pm f(x)$.
Applying \eqref{eq:abscont} to \eqref{eq:generexplicit} (but leaving the $q_3$ term untouched), we see that
\begin{align}
G_X Af(q) =&\ f_1(x) \frac{1}{n} \big( n\lambda 1(q_1 < n) - (q_1 - q_2) \big) + f_2(x) \frac{1}{n} \big( n\lambda 1(q_1 = n, q_2 < n) - q_2 \big) \notag \\
& + n\lambda 1(q_1 < n) \int_{x_1}^{x_1+1/n} (x_1 + 1/n - u) f_{11}(u) du  \notag \\
&+ n\lambda 1(q_1 = n, q_2 < n) \int_{x_2}^{x_2+1/n} (x_2 + 1/n - u) f_{22}(u) du \notag \\
&+ (q_1 - q_2) \int_{x_1-1/n}^{x_1} (u - (x_1 - 1/n)) f_{11}(u) du \notag \\
&+ q_2 \int_{x_2-1/n}^{x_2} (u - (x_2 - 1/n)) f_{22}(u) du \notag \\
& - q_3 \big( f(x_1 , x_2-1/n) - f(x_1,x_2) \big). \label{eq:gen}
\end{align}
To conclude, we rewrite the first line of \eqref{eq:gen} as 
\begin{align*}
& f_1(x) \frac{1}{n} \big( n\lambda - (q_1 - q_2) \big) - f_2(x) \frac{1}{n} q_2  \\
&+ (f_2(x)-f_1(x) ) \lambda 1(q_1 = n)  - f_2(x) \lambda 1(q_1 = q_2 = n) \\
=&\ f_1(x) \big( -\beta/\sqrt{n} - x_1 + x_2 \big) - x_2f_2(x)  \\
&+ (f_2(x)-f_1(x) ) \lambda 1(q_1 = n)  - f_2(x) \lambda 1(q_1 = q_2 = n) \\
=&\ L f(x) +  (f_2(x)-f_1(x) ) \lambda 1(q_1 = n)  - f_2(x) \lambda 1(q_1 = q_2 = n).
\end{align*}

\end{proof}
\subsection{Proving \eqref{eq:thm2} } \label{app:main2}
Our goal is to prove \eqref{eq:thm2}, or that $\E nX_i = \E Q_i \leq C(\beta)$ for all $i \geq 3$. Since $\E Q_i \leq \E Q_3$ for $i \geq 3$, it suffices to consider $i = 3$. Our starting point is \eqref{eq:interm3}, which we recall below:
\begin{align}
\E \big((X_2 - \kappa/\sqrt{n}) \vee 0\big) \leq \frac{1}{\beta\sqrt{n}} \Big( 12  +\frac{6\kappa}{\kappa- \beta }   \Big)\Prob(X_2 \geq \kappa/\sqrt{n} - 1/n). \label{eq:interm4}
\end{align}
Consider $n$ such that $\max(\beta/\sqrt{n},1/n)<1$, and fix $\tilde \kappa  \in ( \max(\beta/\sqrt{n},1/n), 1)$. Invoke \eqref{eq:interm4} with $\sqrt{n} \tilde \kappa$ in place of $\kappa$ there to see that  
\begin{align}
\E \Big((X_2 - \tilde \kappa) 1(X_2 \geq \tilde \kappa)\Big) \leq&\ \frac{1}{\beta\sqrt{n}} \Big(12+\frac{6\tilde \kappa}{\tilde \kappa- \beta/\sqrt{n}}\Big)\Prob(X_2 \geq \tilde \kappa- 1/n) \notag \\
=&\  \frac{1}{\beta n} \Big(12+\frac{6\tilde \kappa}{\tilde \kappa- \beta/\sqrt{n}}\Big)\sqrt{n}\E\Big(\frac{X_2}{X_2} 1(X_2 \geq \tilde \kappa- 1/n)\Big)\notag \\
\leq&\  \frac{1}{\beta n} \Big(12+\frac{6\tilde \kappa}{\tilde \kappa- \beta/\sqrt{n}}\Big)\frac{1}{\tilde \kappa- 1/n} \E \sqrt{n} X_2 \notag \\
\leq&\  \frac{1}{\beta n} \Big(12+\frac{6\tilde \kappa}{\tilde \kappa- \beta/\sqrt{n}}\Big)\frac{1}{\tilde \kappa- 1/n}C(\beta), \label{eq:tailbound}
\end{align}
where in the last inequality we used \eqref{eq:thm1}. Therefore, 
\begin{align*}
\frac{1}{\beta n} \Big(12+\frac{6\tilde \kappa}{\tilde \kappa- \beta/\sqrt{n}}\Big)\frac{1}{\tilde \kappa - 1/n} C(\beta) \geq&\ \E \Big((X_2 - \tilde \kappa) 1(X_2 \geq \tilde \kappa)\Big) \\
\geq&\ (1 - \tilde \kappa)\Prob(X_2 = 1)\\
=&\ (1 - \tilde \kappa)\Prob(Q_2 = n)\\
\geq&\ (1-\tilde \kappa) \frac{1}{n} \E Q_3,
\end{align*}
where in the second inequality we used the fact that $\tilde \kappa < 1$,  and in the last inequality we used Lemma~\ref{lem:q1}.

\begin{remark}
The bound in \eqref{eq:thm2} will be sufficient for our purposes, but it is unlikely to be tight. The argument in \eqref{eq:tailbound} can be modified by observing that for any integer $m > 0$,  
\begin{align*}
\Prob(X_2 \geq \tilde \kappa - 1/n) =&\ \frac{n^m}{n^m} \E\Big(\frac{X_2^{2m}}{X_2^{2m}} 1(X_2 \geq \tilde \kappa - 1/n)\Big) \leq \frac{1}{n^m  (\tilde \kappa  - 1/n)^{2m}} \E (\sqrt{n} X_2)^{2m}.
\end{align*}
Provided we have a bound on $\E (\sqrt{n} X_2)^{2m}$ that is independent of $n$, it follows that $\E Q_3 \leq C(\beta)/n^{m-1/2}$. Although we have not done so, we believe the arguments used in Theorem~\ref{thm:main} can be extended to provide the necessary bounds on $\E (\sqrt{n} X_2)^{2m}$.
\end{remark}

\subsection{Proposition~\ref{thm:interchange}} \label{app:interchange} 
\begin{proof}[Proof of Proposition~\ref{thm:interchange}]
Lemma~\ref{lem:q1} and Theorem~\ref{thm:main} imply that the sequence  $\{\sqrt{n}(X_1,X_2)\}_{n}$ is tight. It follows by Prohorov's Theorem \cite{Bill1999} that the sequence is also relatively compact. We will now show that any subsequence of $\{\sqrt{n}(X_1,X_2)\}_{n}$ has a further subsequence that converges weakly to $Y$.

Fix $n > 0$ and initialize the process $\{X(t)\}_{t \geq 0}$ by letting $\sqrt{n}X(0)$ have the same distribution as $\sqrt{n}X$. Prohorov's Theorem implies that for any subsequence
\begin{align*}
\{\sqrt{n'}X(0)\}_{n'} \subset \{\sqrt{n}X(0)\}_{n},
\end{align*}
there exists a further subsequence
\begin{align*}
\{\sqrt{n''}X(0)\}_{n''} \subset \{\sqrt{n'}X(0)\}_{n'}
\end{align*}
that converges weakly to some random vector $Y^{(0)} = (Y_1^{(0)}, Y_2^{(0)}, \ldots)$. Theorem~\ref{thm:main} implies that $Y_i^{(0)} = 0$ for $i \geq 3$. Now for any $t \geq 0$, let $(Y_1(t),Y_2(t))$ solve the integral equation in \eqref{eq:diffusion} with  intial condition $(Y_1(0),Y_2(0)) = (Y_1^{(0)},Y_2^{(0)})$. Theorem~\ref{thm:transient} says that for any $T> 0$ and $t \in [0,T]$, 
\begin{align}
\{\sqrt{n''}(X_1(t),X_2(t)),\ t \in [0,T]\} \Rightarrow \{(Y_1(t),Y_2(t)),\ t \in [0,T]\} \label{eq:li1} 
\end{align}
as $n \to \infty$, where the convergence is uniform over bounded intervals. Furthermore, since $\{(X_1(t),X_2(t)) \}$ was initialized according to the stationary distribution, 
\begin{align}
\lim_{n \to \infty}\sqrt{n''}(X_1(t),X_2(t)) \stackrel{d}{=} \lim_{n \to \infty}\sqrt{n''}(X_1(0),X_2(0))    =  (Y_1(0),Y_2(0)), \quad  t \in [0,T]. \label{eq:li2}
\end{align}
It follows from \eqref{eq:li1} and \eqref{eq:li2} that 
\begin{align*}
(Y_1(t),Y_2(t)) \stackrel{d}{=}(Y_1(0),Y_2(0)) =  (Y_1^{(0)},Y_2^{(0)}), \quad t \in [0,T],
\end{align*}
meaning $\{(Y_1(t),Y_2(t))\}$ is a stationary process, and must therefore be distributed according to its stationary distribution $(Y_1,Y_2)$. To conclude, we have shown that  $\sqrt{n''}(X_1,X_2)$ converges in distribution to $(Y_1,Y_2)$, which implies convergence of the original sequence $\sqrt{n}(X_1,X_2)$.
\end{proof}

\section{Technical lemmas: Section~\ref{sec:derbounds}. }
In this appendix we prove the key technical lemmas from Section~\ref{sec:derbounds}. Section~\ref{sec:lambertw} has the proofs for Lemmas~\ref{lem:gamma} and \ref{lem:tau} and Section~\ref{app:dertech} has the proof for Lemma~\ref{lem:derivs}. 
\subsection{Lemmas in Section~\ref{sec:tau}.} \label{sec:lambertw}
 A function known as the Lambert W function will play a central role here; the following discussion is based on \cite{CorlGonnHareJeffKnut1996}.   Define $W(x)$ as the solution to 
\begin{align}
x = W(x) e^{W(x)}, \quad x \in [-e^{-1},\infty). \label{eq:lambertdef}
\end{align} 
The function $W(x)$ exists and is known as the Lambert W function. Taking logarithms on both sides of \eqref{eq:lambertdef}, 
\begin{align}
W(x) = \log x - \log W(x). \label{eq:loglambert}
\end{align} 
As is depicted in the plot of $W(x)$ in Figure~\ref{fig:lambert}, $W(-e^{-1}) = -1$, $W(0) = 0$, and $W(x) \to \infty$ as $x \to \infty$. Furthermore, $W(x)$ is multi-valued for $x \in (-e^{-1},0)$, where it is separated into two `branches' $W_0(x)$ and $W_{-1}(x)$; the former is commonly called the principal branch.
\begin{figure}
\begin{center}
\includegraphics[scale=.4]{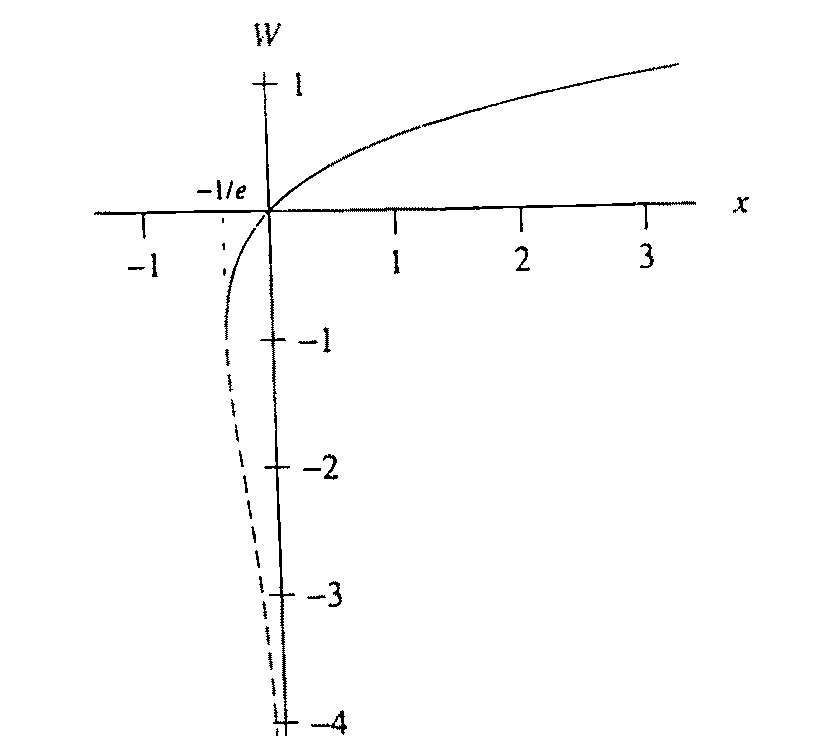}
\end{center}
\caption{A plot of $W(x)$ taken from \cite{CorlGonnHareJeffKnut1996}. For $x \leq 0$, the dashed line represents $W_{-1}(x)$ and the solid line represents $W_0(x)$.} \label{fig:lambert}
\end{figure}
 We will also need to use the fact that $W(x)$ and $W_0(x)$ are differentiable for $x > 0$ and $x \in (-e^{-1},0)$, respectively, and that 
\begin{align}
W'(x) = \frac{W(x)}{x(1+W(x))} > 0, \quad x \in (-e^{-1},0) \cup (0,\infty); \label{eq:wder}
\end{align}
 c.f.\ section 3 of \cite{CorlGonnHareJeffKnut1996}. Going forward, we adopt the convention of using $W(x)$ to mean $W_{0}(x)$ for negative values of $x$. A useful property of $W(x)$ is that
\begin{align}
x = W(xe^{x}), \quad x \geq -1.  \label{eq:lambertprop}
\end{align}
This can be seen by applying $W(x)$ to both sides of \eqref{eq:lambertdef} and using the fact that the range of $W(x)$ is $[-1,\infty]$ (again, we are using the convention $W(x) = W_0(x)$ for $x \in (-e^{-1},0)$).  Furthermore, $W(x)$ is invertible; indeed, $W^{-1}(x) = x e^{x}$ due to \eqref{eq:lambertprop}.

\subsubsection{Proving Lemma~\ref{lem:gamma}.} \label{app:gamma}
We first prove a technical result about $W(x)$, and then prove Lemma~\ref{lem:gamma}. 
\begin{lemma} \label{lem:uniqy}
Fix $\kappa \geq \beta$ and  $x_1 \leq 0$. The  equation 
\begin{align}
W\Big( \frac{-\beta/\sqrt{n}}{\nu} e^{\frac{-(x_1+\beta/\sqrt{n})}{\nu}} \Big) = -\frac{\beta}{\kappa} \label{eq:boundary}
\end{align}
has a unique solution $\nu^* \geq \kappa/\sqrt{n}$. Furthermore,
\begin{align}
\frac{d}{d \nu} \bigg(\frac{-\beta/\sqrt{n}}{\nu} e^{\frac{-(x_1+\beta/\sqrt{n})}{\nu}} \bigg) > 0, \quad \nu \geq \nu^*. \label{eq:posderiv}
\end{align}
\end{lemma}

\begin{proof}[Proof of Lemma~\ref{lem:uniqy}  ] 
Let 
\begin{align*}
f(\nu) = \frac{-\beta/\sqrt{n}}{\nu} e^{\frac{-(x_1+\beta/\sqrt{n})}{\nu}}.
\end{align*}
Since $\kappa \geq \beta$ and $W(x)$ is one-to-one (recall our convention that $W(x) = W_0(x)$  for $x \leq 0$), \eqref{eq:lambertprop} implies that \eqref{eq:boundary} is satisfied if and only if 
\begin{align}
f(\nu) = -\frac{\beta}{\kappa}e^{-\beta/\kappa}. \label{eq:equivalent}
\end{align}
If $x_1 = 0$, then  $\nu = \kappa/\sqrt{n}$ is the unique solution, and so we assume that $x_1 < 0$ and argue that \eqref{eq:equivalent} has a unique solution.  Since the domain of $W(x)$ is $[-e^{-1},\infty)$, we can only consider $\nu$ large enough such that $f(\nu) \geq -1$.
Observe that 
\begin{align*}
f(\kappa/\sqrt{n}) = -\frac{\beta }{\kappa} e^{-\beta/\kappa} e^{-x_1 \sqrt{n}/\kappa} \leq -\frac{\beta }{\kappa} e^{-\beta/\kappa}.
\end{align*} 
  Differentiating, 
\begin{align*}
f'(\nu) =&\ \frac{\beta/\sqrt{n}}{\nu^2}e^{\frac{-(x_1+\beta/\sqrt{n})}{\nu}} - \frac{\beta/\sqrt{n}}{\nu}e^{\frac{-(x_1+\beta/\sqrt{n})}{\nu}}\Big( \frac{x_1+\beta/\sqrt{n}}{\nu^2}\Big)\\
=&\ \frac{\beta/\sqrt{n}}{\nu^2}e^{\frac{-(x_1+\beta/\sqrt{n})}{\nu}}\Big( 1 -  \frac{x_1+\beta/\sqrt{n}}{\nu }\Big).
\end{align*} 
We know that $f(\nu) \to 0$ as $\nu \to \infty$. 

\textbf{Case 1:} $x_1 + \beta/\sqrt{n} \leq 0$. In this case $f'(\nu) > 0$ for all $\nu > 0$, which implies that there exists a unique $\nu^{*}$ such that \eqref{eq:equivalent} is satisfied.

\textbf{Case 2:} $x_1 + \beta/\sqrt{n} \geq 0$. The form of $f'(\nu)$ tells us that $f(\nu)$ is decreasing on $(0,x_1+\beta/\sqrt{n})$, but starts increasing after that. This again implies that a unique $\nu^{*}$ exists, and that $\nu^{*} \geq x_1+\beta/\sqrt{n}$, implying $f'(\nu) > 0$ for all $\nu \geq \nu^*$.

\end{proof}

\begin{proof}[Proof of Lemma~\ref{lem:gamma}]
Fix $\kappa \geq \beta$ and  $x_1 \leq 0$. We begin by showing that the system \eqref{eq:nonlin} has a unique solution. The first step is to write $\eta$ in terms of $\nu$.  We rearrange 
\begin{align}
-\beta/\sqrt{n} + (x_1 + \beta/\sqrt{n}) e^{-\eta} +  \eta \nu e^{-\eta} = 0 \label{eq:step1}
\end{align}
into 
\begin{align*}
\frac{ -\beta/\sqrt{n}}{\nu}  =&\  \frac{-(x_1+\beta/\sqrt{n})}{\nu} e^{-\eta} - \eta e^{-\eta} =\Big( - \frac{(x_1+\beta/\sqrt{n})}{\nu}  - \eta \Big)e^{-\eta } ,
\end{align*}
or
\begin{align}
 \Big( - \frac{(x_1+\beta/\sqrt{n})}{\nu}  - \eta \Big)e^{\frac{-(x_1+\beta/\sqrt{n})}{\nu} - \eta } =&\ \frac{-\beta/\sqrt{n}}{\nu} e^{\frac{-(x_1+\beta/\sqrt{n})}{\nu} }.   \label{eq:preW}
\end{align}
Observe that the left hand side of \eqref{eq:preW} is in the form $-xe^{-x}$, and so must lie in $[-e^{-1},\infty)$. Therefore, existence of a solution to \eqref{eq:nonlin} imposes a natural constraint on $\nu$ that the right hand side above must lie in $[-e^{-1},\infty)$. Assuming this is the case, we apply $W(x)$ to both sides of \eqref{eq:preW} and using \eqref{eq:lambertprop}, we arrive at 
\begin{align}
\eta =&\  - \frac{(x_1+\beta/\sqrt{n})}{\nu} - W\bigg(\frac{-\beta/\sqrt{n}}{\nu} e^{\frac{-(x_1+\beta/\sqrt{n})}{\nu} }\bigg). \label{eq:eta}
\end{align}
Since the Lambert W function is multivalued for $x \leq 0$, the above equation tells us that given $\nu$, there can be two potential choices for $\eta$. Plugging the above form of $\eta$ back into \eqref{eq:step1}, we see that 
\begin{align*}
-\beta/\sqrt{n}  - W\bigg(\frac{-\beta/\sqrt{n}}{\nu} e^{\frac{-(x_1+\beta/\sqrt{n})}{\nu} }\bigg)\nu e^{-\eta} = 0,
\end{align*}
which, after using the fact that $\nu e^{-\eta} = \kappa/\sqrt{n}$, becomes
\begin{align}
W\bigg(\frac{-\beta/\sqrt{n}}{\nu} e^{\frac{-(x_1+\beta/\sqrt{n})}{\nu} }\bigg) = -\frac{\beta}{\kappa}. \label{eq:step2}
\end{align}
Lemma~\ref{lem:uniqy} tells us that \eqref{eq:step2} does indeed have a unique solution $\nu^*(x_1)$. From \eqref{eq:eta} we know $\eta$ can have up to two values, but we narrow this number down to one using the fact that $\nu e^{-\eta} = \kappa/\sqrt{n}$. As an aside, it can be verified that $\nu^*(0) = \kappa/\sqrt{n}$, and $\eta^*(0) = 0$.

We now prove the second claim in the lemma that $\gamma^{(\kappa)}(x_1) \subset \Gamma^{(\kappa)}$ for every $x_1 \leq 0$. Recall that 
\begin{align*}
 \gamma^{(\kappa)}(x_1) =  \Big\{ \big(-\beta/\sqrt{n} + (x_1 + \beta/\sqrt{n}) e^{-t} + t \nu^*(x_1) e^{-t}, \nu^*(x_1) e^{-t} \big) \ \Big| \ t \in [0,\eta^*(x_1)] \Big\}.
\end{align*}
Given $x_1 \leq 0$ and $t \in [0,\eta^*(x_1)]$, define 
\begin{align*}
\bar x_1 = -\beta/\sqrt{n} + (x_1 + \beta/\sqrt{n}) e^{-t} + t \nu^*(x_1) e^{-t}.
\end{align*}
By uniqueness of $\nu^*(\bar x_1)$ and $\eta^*(\bar x_1)$, it suffices to show that the pair
\begin{align*}
\nu = \nu^*(x_1) e^{-t}, \quad \eta = \eta^*(x_1) - t
\end{align*}
solves \eqref{eq:nonlin} with $\bar x_1$ replacing $x_1$ there. Indeed, 
\begin{align*}
\nu^*(x_1) e^{-t} e^{-(\eta^*(x_1) - t)} = \nu^*(x_1) e^{-\eta^*(x_1)} = \kappa/\sqrt{n},
\end{align*}
and 
\begin{align*}
& -\beta/\sqrt{n} + (\bar x_1 + \beta/\sqrt{n}) e^{-(\eta^*(x_1) - t)} + (\eta^*(x_1) - t) \nu^*(x_1) e^{-t} e^{-(\eta^*(x_1) - t)}\\
=&\ -\beta/\sqrt{n} +(\bar x_1 + \beta/\sqrt{n}) e^{-(\eta^*(x_1) - t)} + (\eta^*(x_1) - t) \nu^*(x_1) e^{-\eta^*(x_1)}\\
=&\ -\beta/\sqrt{n} + (  (x_1 + \beta/\sqrt{n}) e^{-t} + t \nu^*(x_1) e^{-t} ) e^{-(\eta^*(x_1) - t)}  + (\eta^*(x_1) - t) \nu^*(x_1) e^{-\eta^*(x_1)}\\
=&\ -\beta/\sqrt{n} + (x_1 + \beta/\sqrt{n}) e^{-\eta^*(x_1) } + \eta^*(x_1) \nu^*(x_1) e^{-\eta^*(x_1)} = 0.
\end{align*}
\end{proof}

\subsubsection{Proving Lemma~\ref{lem:tau}. } \label{app:tau}

\begin{proof}[Proof of Lemma~\ref{lem:tau} ] 
Fix $\kappa \geq \beta$ and $x \in \Omega$. Assume that $x \geq \Gamma^{(\kappa)}$, which by definition in \eqref{eq:xgeqgamma} implies that there exists some $\delta \geq 0$ such that 
\begin{align*}
(x_1, x_2 - \delta) \in \Gamma^{(\kappa)}.
\end{align*}
Note that if $x > \Gamma^{(\kappa)}$, then $\delta > 0$. 

Let us prove \eqref{eq:gammaprop}. Consider the equation
\begin{align}
\beta/\sqrt{n} - (x_1 + \beta/\sqrt{n}) e^{-\eta} - \eta x_2 e^{-\eta} = 0. \label{eq:interm}
\end{align}
We first argue that 
\begin{align}
\eta =&\  - \frac{(x_1+\beta/\sqrt{n})}{x_2} - W\bigg(\frac{-\beta/\sqrt{n}}{x_2} e^{\frac{-(x_1+\beta/\sqrt{n})}{x_2} }\bigg). \label{eq:etaaltern}
\end{align}
Starting with \eqref{eq:interm}, we can replicate the steps used to get \eqref{eq:preW} to see that \eqref{eq:interm} is equivalent to 
\begin{align}
 \Big( - \frac{(x_1+\beta/\sqrt{n})}{x_2}  - \eta \Big)e^{\frac{-(x_1+\beta/\sqrt{n})}{x_2} - \eta } =&\ \frac{-\beta/\sqrt{n}}{x_2} e^{\frac{-(x_1+\beta/\sqrt{n})}{x_2} }. \label{eq:preW2}
\end{align}
Let us assume that $x \geq \Gamma^{(\kappa)}$ implies that the right hand side of \eqref{eq:preW2} is in the interval $[-e^{-1},0)$; we postpone the verification of this claim for now. We can apply $W(\cdot)$ to both sides of \eqref{eq:preW2} and use \eqref{eq:lambertprop}  to conclude \eqref{eq:etaaltern}. Plugging \eqref{eq:etaaltern} back into \eqref{eq:interm},
\begin{align}
-\beta/\sqrt{n}    - W\bigg(\frac{-\beta/\sqrt{n}}{x_2} e^{\frac{-(x_1+\beta/\sqrt{n})}{x_2} }\bigg)x_2 e^{-\eta} = 0, \label{eq:etaaltern2}
\end{align}
or 
\begin{align*}
x_2 e^{-\eta} = \frac{-\beta/\sqrt{n}}{W\bigg(\frac{-\beta/\sqrt{n}}{x_2} e^{\frac{-(x_1+\beta/\sqrt{n})}{x_2} }\bigg)} \geq \frac{-\beta/\sqrt{n}}{W\bigg(\frac{-\beta/\sqrt{n}}{x_2-\delta} e^{\frac{-(x_1+\beta/\sqrt{n})}{x_2-\delta} }\bigg)} = \kappa/\sqrt{n}.
\end{align*}
Observe that the inequality above is strict if $x > \Gamma^{(\kappa)}$, and that it becomes an equality if $\delta = 0$ (which means that $x \in \Gamma^{(\kappa)}$).   

To conclude the proof of \eqref{eq:gammaprop}, it remains verify our assumption that $x \geq \Gamma^{(\kappa)}$ implies that the right hand side of \eqref{eq:preW2} is in the interval $[-e^{-1},0)$.  Equation \eqref{eq:step2} in the proof of Lemma~\ref{lem:gamma} tells us that $(x_1, x_2 - \delta) \in \Gamma^{(\kappa)}$ implies 
\begin{align*}
W\bigg(\frac{-\beta/\sqrt{n}}{x_2 - \delta} e^{\frac{-(x_1+\beta/\sqrt{n})}{x_2 - \delta} }\bigg) = -\frac{\beta}{\kappa},
\end{align*}
or that 
\begin{align}
\frac{-\beta/\sqrt{n}}{x_2 - \delta} e^{\frac{-(x_1+\beta/\sqrt{n})}{x_2 - \delta} } = W^{-1}(-\beta/\kappa) \geq W^{-1}(-1) = -e^{-1}, \label{eq:interm2}
\end{align}
where in the inequality above we used the fact that $\kappa \geq \beta$ and that $W^{-1}(\cdot)$ is an increasing function.  Now from \eqref{eq:posderiv} we know that 
\begin{align*}
\frac{d}{d \nu} \bigg(\frac{-\beta/\sqrt{n}}{\nu} e^{\frac{-(x_1+\beta/\sqrt{n})}{\nu}} \bigg) > 0, \quad \nu \geq x_2 - \delta,
\end{align*}
which implies 
\begin{align*}
\frac{-\beta/\sqrt{n}}{x_2} e^{\frac{-(x_1+\beta/\sqrt{n})}{x_2} } \geq \frac{-\beta/\sqrt{n}}{x_2 - \delta} e^{\frac{-(x_1+\beta/\sqrt{n})}{x_2 - \delta} } = W^{-1}(-\beta/\kappa) \geq -e^{-1}.
\end{align*}
This concludes the proof of \eqref{eq:gammaprop}.

 We now address the differentiability of $\tau(x)$ to prove \eqref{eq:tauder}. Fixing $\kappa > \beta$ and $x \geq \Gamma^{(\kappa)}$, we see from \eqref{eq:etaaltern} that
\begin{align*}
\tau(x) = - \frac{(x_1+\beta/\sqrt{n})}{x_2} - W\bigg(\frac{-\beta/\sqrt{n}}{x_2} e^{\frac{-(x_1+\beta/\sqrt{n})}{x_2} }\bigg).
\end{align*}
We know that $W'(u)$ exists for $u \in (-e^{-1},0)$,  and that
\begin{align*}
\frac{-\beta/\sqrt{n}}{x_2} e^{\frac{-(x_1+\beta/\sqrt{n})}{x_2} } > -e^{-1}, \quad x \geq \Gamma^{(\kappa)},
\end{align*}
which can be derived from \eqref{eq:interm2}. Therefore, $\tau(x)$ is differentiable at all points $x \geq \Gamma^{(\kappa)}$ with $x_1 < 0$. Only the one-sided derivative exists for $x \in \{x_1 = 0,\ x \geq \Gamma^{(\kappa)}\}$, i.e.\ those $x$ that are on the vertical axis. To characterize the derivatives of $\tau(x)$, let us use the form
\begin{align*}
\tau(x) =  - \frac{(x_1+\beta/\sqrt{n})}{x_2} + \frac{\beta/\sqrt{n}}{x_2 e^{-\tau(x)}},
\end{align*}
which is implied by \eqref{eq:etaaltern} and \eqref{eq:etaaltern2}.
Differentiating gives us 
\begin{align}
\tau_1(x) =&\ -\frac{1}{x_2} \Big( 1 -  \frac{\beta/\sqrt{n}}{x_2 e^{-\tau(x)}} \Big)^{-1} = -\frac{1}{x_2} \frac{x_2 e^{-\tau(x)}}{x_2 e^{-\tau(x)} - \beta/\sqrt{n}} = - \frac{ e^{-\tau(x)}}{x_2 e^{-\tau(x)} - \beta/\sqrt{n}}, \label{eq:exampletau1}
\end{align}
where $\tau_1(x)$ is understood to be the left derivative when $x_1 = 0$. Note that $x \geq \Gamma^{(\kappa)}$ means the denominator in $\tau_1(x)$ is strictly positive due to our recently proved \eqref{eq:gammaprop}. Furthermore, 
\begin{align*}
\tau_2(x) = \frac{x_1 + \beta/\sqrt{n}}{x_2^2} - \frac{\beta/\sqrt{n}}{x_2^2 e^{-\tau(x)}} + \tau_2(x) \frac{\beta/\sqrt{n}}{x_2 e^{-\tau(x)}} = -\frac{1}{x_2} \tau(x) + \tau_2(x)\frac{\beta/\sqrt{n}}{x_2 e^{-\tau(x)}} ,
\end{align*}
and so 
\begin{align}
\tau_2(x) = -\frac{1}{x_2} \tau(x) \Big( 1 -  \frac{\beta/\sqrt{n}}{x_2 e^{-\tau(x)}} \Big)^{-1} = \tau_1(x) \tau(x). \label{eq:exampletau2}
\end{align}
This proves \eqref{eq:tauder}, and we now prove the last claim of the lemma. 
Fix $x = (x_1,x_2)$ and assume that $x \geq \Gamma^{(\kappa_2)}$. By \eqref{eq:gammaprop}, we know that $x_2 e^{-\tau(x)} \geq \kappa_2/\sqrt{n} > \kappa_1/\sqrt{n}$. Now
\begin{align*}
\frac{d}{d x_2} x_2 e^{-\tau(x)} =  e^{-\tau(x)} - \tau_2(x) x_2  e^{-\tau(x)} > 0, \quad x \geq \Gamma^{(\kappa_2)},
\end{align*}
where the inequality follows from the form of $\tau_2(x)$ in \eqref{eq:tauder}. Therefore,
\begin{align*}
(x_2 + \varepsilon) e^{-\tau(x_1, x_2+\varepsilon)} \geq \kappa_2/\sqrt{n} > \kappa_1/\sqrt{n}, \quad \varepsilon \geq 0.
\end{align*}
In other words, $(x_1,x_2+\varepsilon) \not \in \Gamma^{(\kappa_1)}$ for all $\varepsilon \geq 0$ by definition of $\Gamma^{(\kappa_1)}$ in Lemma~\ref{lem:gamma}.  However, also by Lemma~\ref{lem:gamma}, there must exist some $\bar x_2 \geq 0$ such that $(x_1,\bar x_2) \in \Gamma^{(\kappa_1)}$, which means that $\bar x_2 = x_2 - \bar \varepsilon$ for some $\bar \varepsilon > 0$, or that $x > \Gamma^{(\kappa_1)}$.
\end{proof}

\subsection{Lemmas in Section~\ref{sec:candidate}. } \label{app:dertech}
\begin{proof}[Proof of Lemma~\ref{lem:derivs}]
The proof proceeds as follows. We first show that  $f_1(\cdot,x_2), f_{2}(x_1,\cdot)$ are absolutely continuous for  all $x \in \Omega$. We then verify that $f^*(x)$ satisfies the PDE \eqref{eq:poisson1} with the boundary condition \eqref{eq:poisson2}. Lastly, we verify the bounds on the second derivatives of $f^*(x)$.

\subsubsection{First Derivatives}
Recall the definition of $f^*(x)$:
\begin{align*}
f^*(x) = 
\begin{cases}
&0, \hfill x_2 \in [0,\kappa/\sqrt{n}],\\
&x_2- \frac{\kappa}{ \sqrt{n}}  - \frac{\kappa}{\sqrt{n}} \log (\sqrt{n}x_2/\kappa),  \hfill  x \leq \Gamma^{(\kappa)} \text{ and } x_2 \geq \kappa/\sqrt{n}, \\
&x_2(1 - e^{-\tau(x)}) - \frac{\kappa}{\sqrt{n}} \tau(x) +   \frac{1}{2} \frac{\sqrt{n}}{\beta} (x_2 e^{-\tau(x)} - \kappa/\sqrt{n})^2,  \quad  x \geq \Gamma^{(\kappa)}.
\end{cases}
\end{align*}
Let us differentiate $f^*(x)$ in the region $x \leq \Gamma^{(\kappa)}$ and $x_2 \geq \kappa/\sqrt{n}$:
\begin{align}
f_{1}^*(x) = 0, \quad f_2^*(x) = 1 - \frac{\kappa}{x_2 \sqrt{n}}, \quad f_{22}^*(x) = \frac{1}{x_2^2} \frac{\kappa}{\sqrt{n}}, \quad x \leq \Gamma^{(\kappa)} \text{ and } x_2 \geq \kappa/\sqrt{n}. \label{eq:d0}
\end{align}
Next, we differentiate $f^*(x)$ when $x \geq \Gamma^{(\kappa)}$:
\begin{align}
f_1^*(x) =&\ \tau_1(x) x_2 e^{-\tau(x)} - \frac{\kappa}{\sqrt{n}} \tau_1(x) + \frac{\sqrt{n}}{\beta}  (x_2 e^{-\tau(x)} - \kappa/\sqrt{n}) (-x_2 \tau_1(x) e^{-\tau(x)} ) \notag \\
=&\ -\tau_1(x)\Big( -x_2 e^{-\tau(x)} + \frac{\kappa}{\sqrt{n}} +x_2^2   e^{-2\tau(x)} \frac{\sqrt{n}}{\beta}   -x_2   e^{-\tau(x)} \frac{\kappa}{\beta}  \Big) \notag \\
=&\ -\tau_1(x)\frac{\sqrt{n}}{\beta} \Big(x_2^2   e^{-2\tau(x)}   -\frac{\beta}{\sqrt{n}} x_2 e^{-\tau(x)}   -x_2   e^{-\tau(x)} \frac{\kappa}{\sqrt{n}} + \frac{\kappa \beta}{n}   \Big) \notag \\
=&\ -\tau_1(x)\frac{\sqrt{n}}{\beta}  \big(x_2  e^{-\tau(x)}   -\beta/\sqrt{n} \big)\big(   x_2  e^{-\tau(x)}   -\kappa/\sqrt{n}  \big) \notag \\
=&\ \frac{\sqrt{n}}{\beta} e^{-\tau(x)} \big(   x_2  e^{-\tau(x)}   -\kappa/\sqrt{n}  \big), \label{eq:d1}
\end{align}
where in the last equality we used \eqref{eq:tauder}. Now we will prove that 
\begin{align}
f_2^*(x) =  1 - \frac{1}{x_2} \frac{\kappa}{\sqrt{n}}    + \frac{\sqrt{n}}{\beta}\big(   x_2  e^{-\tau(x)} - \kappa/\sqrt{n}\big) \Big( \frac{x_2  e^{-\tau(x)}   -\beta/\sqrt{n}}{x_2} + \tau e^{-\tau(x)} \Big), \quad x \geq \Gamma^{(\kappa)}. \label{eq:d2}
\end{align}
We begin by differentiating $f^*(x)$:
\begin{align*}
f_2^*(x) =&\ (1 - e^{-\tau(x)})+ \tau_2(x) x_2 e^{-\tau(x)}  - \frac{\kappa}{\sqrt{n}} \tau_{2}(x) \\
&+    \frac{\sqrt{n}}{\beta} (x_2 e^{-\tau(x)} - \kappa/\sqrt{n}) (e^{-\tau(x)} - \tau_{2}(x) x_2 e^{-\tau(x)}) \\
=&\ (1 - e^{-\tau(x)})+ \tau_{2}(x) x_2 e^{-\tau(x)}  - \frac{\kappa}{\sqrt{n}} \tau_{2}(x) \\
&+    \frac{\sqrt{n}}{\beta} \big(x_2 e^{-2\tau(x)} - \tau_{2}(x) x_2^2 e^{-2\tau(x)} - e^{-\tau(x)}\kappa/\sqrt{n}  + \tau_{2}(x) x_2 e^{-\tau(x)}\kappa/\sqrt{n}\big),
\end{align*}
which equals 
\begin{align}
& (1 - e^{-\tau(x)}) + \frac{\sqrt{n}}{\beta}x_2 e^{-2\tau(x)} - e^{-\tau(x)}\frac{\kappa}{\beta} \label{eq:m1} \\
&- \tau_{2}(x) \frac{\sqrt{n}}{\beta} \big( x_2^2 e^{-2\tau(x)} - \frac{\beta}{\sqrt{n}} x_2 e^{-\tau(x)}   - \frac{\kappa}{\sqrt{n}} x_2 e^{-\tau(x)} + \frac{\kappa \beta}{n} \big). \label{eq:m2}
\end{align}
We focus on \eqref{eq:m1}, which equals
\begin{align*}
& 1 + \frac{\sqrt{n}}{\beta} \frac{1}{x_2}\Big(x_2^2 e^{-2\tau(x)} - x_2e^{-\tau(x)}\frac{\kappa}{\sqrt{n}}- x_2e^{-\tau(x)}\frac{\beta}{\sqrt{n}} \Big)  \\
=&\ 1 - \frac{1}{x_2} \frac{\kappa}{\sqrt{n}} + \frac{\sqrt{n}}{\beta} \frac{1}{x_2}\Big(x_2^2 e^{-2\tau(x)} - x_2e^{-\tau(x)}\frac{\kappa}{\sqrt{n}}- x_2e^{-\tau(x)}\frac{\beta}{\sqrt{n}} + \frac{\kappa \beta}{n} \Big) \\
=&\ 1 - \frac{1}{x_2} \frac{\kappa}{\sqrt{n}} + \frac{\sqrt{n}}{\beta} \frac{1}{x_2}\big(x_2  e^{-\tau(x)}   -\beta/\sqrt{n} \big)\big(   x_2  e^{-\tau(x)} - \kappa/\sqrt{n}\big).  
\end{align*}
With the help of \eqref{eq:tauder}, we see that \eqref{eq:m2}  equals 
\begin{align*}
& - \tau_{2}(x) \frac{\sqrt{n}}{\beta} \big(x_2  e^{-\tau(x)}   -\beta/\sqrt{n} \big)\big(   x_2  e^{-\tau(x)} - \kappa/\sqrt{n}\big)= \frac{\sqrt{n}}{\beta}\tau e^{-\tau(x)}  \big(   x_2  e^{-\tau(x)} - \kappa/\sqrt{n}\big).
\end{align*}
Therefore, for all $x \geq \Gamma^{(\kappa)}$,
\begin{align*}
f_2^*(x) =  1 - \frac{1}{x_2} \frac{\kappa}{\sqrt{n}}    + \frac{\sqrt{n}}{\beta}\big(   x_2  e^{-\tau(x)} - \kappa/\sqrt{n}\big) \Big( \frac{x_2  e^{-\tau(x)}   -\beta/\sqrt{n}}{x_2} + \tau e^{-\tau(x)} \Big).
\end{align*}

We now verify continuity of the partial derivatives of $f^*(x)$. 
Recall that the support of $f^*(x)$ is naturally partitioned into three subdomains:
\begin{align}
\{x_2 \in [0,\kappa/\sqrt{n}]\}, \quad  \{x_2 \leq \Gamma^{(\kappa)},\ x_2 \geq \kappa/\sqrt{n} \}, \quad \text{ and } \quad \{x_2 \geq \Gamma^{(\kappa)} \}. \label{eq:subdomain}
\end{align}
Continuity of the partial derivatives on the interiors of these subdomains follows from the continuity of $\tau(x)$, and it remains to verify continuity on the intersections, which are 
\begin{align*}
\{x_2 \in [0,\kappa/\sqrt{n}]\} \cap \{x \leq \Gamma^{(\kappa)},\ x_2 \geq \kappa/\sqrt{n} \} =&\ \{ x_2 = \kappa/\sqrt{n}\},\\
\{x \leq \Gamma^{(\kappa)},\ x_2 \geq \kappa/\sqrt{n} \} \cap  \{x \geq \Gamma^{(\kappa)} \} =&\ \{x \in \Gamma^{(\kappa)}\},\\
\{x_2 \in [0,\kappa/\sqrt{n}]\} \cap  \{x \geq \Gamma^{(\kappa)} \} =&\ \{(0,\kappa/\sqrt{n})\}\subset \Gamma^{(\kappa)}.
\end{align*}
The fact that $\{(0,\kappa/\sqrt{n})\}\subset \Gamma^{(\kappa)}$ follows from the definition of $\Gamma^{(\kappa)}$ in Lemma~\ref{lem:gamma}.
 When $x_2 = \kappa/\sqrt{n}$, we see from \eqref{eq:d0} that $f_1^*(x) = f_2^*(x) = 0$, which confirms continuity on $\{ x_2 = \kappa/\sqrt{n}\}$. Now by definition, $x \in \Gamma^{(\kappa)}$ implies that $x_2 e^{-\tau(x)} = \kappa/\sqrt{n}$, from which we see that \eqref{eq:d0} coincides with \eqref{eq:d1}-\eqref{eq:d2}. 
 
Thus we have proved continuity of the derivatives of $f^*(x)$ on $\Omega$. It remains to prove that  $f_1^*(\cdot,x_2), f_{2}^*(x_1,\cdot)$ are absolutely continuous for  all $x \in \Omega$. 

Fix $x_2 \geq 0$. We will show that $f_1^*(\cdot,x_2)$ is differentiable almost everywhere. From the form of $f^*(x)$ and \eqref{eq:d0}, we see that $\frac{d}{d x_1} f_{1}^*(x)= 0$ on the set $\{x_1 \leq 0 :\ x < \Gamma^{(\kappa)}\}$. Furthermore, from \eqref{eq:d1} we know that $\frac{d}{d x_1}f_1^*(x)$ exists on the interior of $\{x_1 \leq 0 :\ x \geq \Gamma^{(\kappa)}\}$ (because $\tau(x)$ is differentiable). It may be that $\frac{d}{d x_1}f_{1}^*(x)$ does not exist on the set $\{x_1:\ x \in \Gamma^{(\kappa)}\}$. However, Lemma~\ref{lem:gamma} (and in particular the form of $\gamma^{(\kappa)}(x_1)$) tells us that the curve $\Gamma^{(\kappa)}$ contains no horizontal segments because the second coordinate of $\gamma^{(\kappa)}(x_1)$ is always decreasing with $t$. Therefore, the set $\{x_1:\ x \in \Gamma^{(\kappa)}\}$ contains at most one point, and for each fixed $x_2 \geq 0$, the function $f_1^*(\cdot, x_2)$ is differentiable almost everywhere, and therefore absolutely continuous.

A similar argument holds for showing $f^*_2(x_1,\cdot)$ is absolutely continuous. For fixed $x_1 \leq 0$, the function $f^*_2(x_1,\cdot)$ is differentiable everywhere except the point $x_2 = \kappa/\sqrt{n}$ and the set $\{x_2:\ x \in \Gamma^{(\kappa)}\}$. However, the latter contains only a single point because given $x_1$, Lemma~\ref{lem:gamma} (namely, uniqueness of $\nu^*(x_1)$) tells us the set $\{x_2:\ x \in \Gamma^{(\kappa)}\}$ contains only a single point (i.e. $\Gamma^{(\kappa)}$ contains no vertical lines). Therefore, $f^*_2(x_1,\cdot)$ is differentiable almost everywhere and is therefore absolutely continuous.

\subsubsection{Satisfying the PDE}
We now verify that $f^*(x)$ satisfies the PDE \eqref{eq:poisson1} and the boundary condition \eqref{eq:poisson2}. For $x_2 \in [0,\kappa/\sqrt{n}]$, 
\begin{align*}
L f^*(x) = 0 = -\big((x_2 - \kappa/\sqrt{n}) \vee 0\big),
\end{align*}
and 
\begin{align*}
f_1^*(0,x_2) = f_2^*(0,x_2) = 0,
\end{align*}
and so both \eqref{eq:poisson1} and \eqref{eq:poisson2} are trivially satisfied. When $x \leq \Gamma^{(\kappa)}$ and $x_2 \geq \kappa/\sqrt{n}$,
\begin{align*}
L f^*(x) = (-x_1+x_2-\beta/\sqrt{n}) f^*_1(x) - x_2 f^*_2(x) =  -x_2\big(1 - \frac{\kappa}{x_2 \sqrt{n}}\big) = -\big((x_2 - \kappa/\sqrt{n}) \vee 0\big),
\end{align*}
and the only intersection of $x \leq \Gamma^{(\kappa)}$ and $x_2 \geq \kappa/\sqrt{n}$ with the vertical axis is the point $(0,\kappa/\sqrt{n})$, meaning
\begin{align*}
f^*_1(0,x_2) =&\ 0 \\
f_2^*(0,x_2) =&\ 1 - \frac{\kappa}{x_2 \sqrt{n}} = 0.
\end{align*}
The last case is $x \geq \Gamma^{(\kappa)}$. Using \eqref{eq:d1} and \eqref{eq:d2}:
\begin{align*}
L f^*(x) =&\ (-x_1+x_2-\beta/\sqrt{n}) f^*_1(x) - x_2 f^*_2(x) \\
=&\ (-x_1+x_2-\beta/\sqrt{n})\frac{\sqrt{n}}{\beta} e^{-\tau(x)} \big(   x_2  e^{-\tau(x)}   -\kappa/\sqrt{n}  \big) \\
&- x_2 \Big(1 - \frac{1}{x_2} \frac{\kappa}{\sqrt{n}}    + \frac{\sqrt{n}}{\beta}\big(   x_2  e^{-\tau(x)} - \kappa/\sqrt{n}\big) \Big( \frac{x_2  e^{-\tau(x)}   -\beta/\sqrt{n}}{x_2} + \tau e^{-\tau(x)} \Big)\Big)\\
=&\ -\big(x_2 - \frac{\kappa}{\sqrt{n}}\big) + \frac{\sqrt{n}}{\beta}\big(   x_2  e^{-\tau(x)} - \kappa/\sqrt{n}\big)\Big( \beta/\sqrt{n} + x_2 \tau(x) e^{-\tau(x)} + (-x_1  - \beta/\sqrt{n})e^{-\tau(x)} \Big)\\
=&\ -\big((x_2 - \kappa/\sqrt{n}) \vee 0\big),
\end{align*}
where the last equality follows from the definition of $\tau(x)$ in Lemma~\ref{lem:tau}. Verifying the boundary condition:
\begin{align*}
f_1^*(0,x_2) = \frac{\sqrt{n}}{\beta} e^{-\tau(0,x_2)} \big(   x_2  e^{-\tau(0,x_2)}   -\kappa/\sqrt{n}  \big) = \frac{\sqrt{n}}{\beta}  \big(   x_2     -\kappa/\sqrt{n}  \big),
\end{align*}
and 
\begin{align*}
f_2^*(0,x_2) =&\ 1 - \frac{1}{x_2} \frac{\kappa}{\sqrt{n}}    + \frac{\sqrt{n}}{\beta}\big(   x_2  e^{-\tau(0,x_2)} - \kappa/\sqrt{n}\big) \Big( \frac{x_2  e^{-\tau(0,x_2)}   -\beta/\sqrt{n}}{x_2} + \tau(0,x_2) e^{-\tau(0,x_2)} \Big) \\
=&\  1 - \frac{1}{x_2} \frac{\kappa}{\sqrt{n}}    + \frac{\sqrt{n}}{\beta}\frac{1}{x_2}\big(   x_2   - \kappa/\sqrt{n}\big) \Big( x_2  -\beta/\sqrt{n}  \Big) \\
=&\  1 - \frac{1}{x_2} \frac{\kappa}{\sqrt{n}}    + \frac{\sqrt{n}}{\beta}\frac{1}{x_2}\big(   x_2^2 - x_2\kappa/\sqrt{n} - x_2\beta/\sqrt{n} + \frac{\kappa\beta}{n}  \big) \\
=&\  \frac{\sqrt{n}}{\beta}\Big(\frac{\beta}{\sqrt{n}} - \frac{1}{x_2} \frac{\kappa\beta	}{n}\Big)    + \frac{\sqrt{n}}{\beta}\big(   x_2  -  \kappa/\sqrt{n} -  \beta/\sqrt{n} + \frac{\kappa\beta}{n} \frac{1}{x_2}  \big) \\
=&\  \frac{\sqrt{n}}{\beta}(x_2 - \kappa/\sqrt{n}) = f^*_1(0,x_2).
\end{align*}
Therefore, our $f^*(x)$ satisfies \eqref{eq:poisson1}--\eqref{eq:poisson2}. 

\subsubsection{Second Derivatives}
It remains to prove the bounds on the second derivatives \eqref{eq:dpos}--\eqref{eq:d22}. When $x \in \{x_2 \in [0,\kappa/\sqrt{n}]\}$, $f_{11}^*(x)=f_{12}^*(x)=f_{22}^*(x)=0$  and when $x \in \{x \leq \Gamma^{(\kappa)},\ x_2 \geq \kappa/\sqrt{n} \}$, $f_{11}^*(x)=f_{12}^*(x)=0$ and 
\begin{align*}
f_{22}^*(x) = \frac{1}{x_2^2}\frac{\kappa}{\sqrt{n}} \leq \frac{\sqrt{n}}{\kappa} \leq \frac{\sqrt{n}}{\beta},
\end{align*}
which satisfies \eqref{eq:dpos}--\eqref{eq:d22}.

Therefore, we are left to deal with the case when $x \geq \Gamma^{(\kappa)}$.  First we take on $f_{11}^*(x)$. Differentiating \eqref{eq:d1} and using \eqref{eq:tauder} one arrives at
\begin{align*}
f_{11}^*(x) =&\ \frac{\sqrt{n}}{\beta} e^{-2\tau(x)} \frac{x_2e^{-\tau(x)} +  \big( x_2e^{-\tau(x)}  -\kappa/\sqrt{n} \big)}{ \big( x_2e^{-\tau(x)}  -\beta/\sqrt{n} \big)},
\end{align*}
from which we conclude that
\begin{align*}
0 \leq f_{11}^*(x) \leq \frac{\sqrt{n}}{\beta}\Big( \frac{x_2e^{-\tau(x)}  }{   x_2e^{-\tau(x)}  -\beta/\sqrt{n}  } + 1 \Big) =&\ \frac{\sqrt{n}}{\beta}\Big( \frac{1}{   1  -\frac{\beta/\sqrt{n}}{x_2e^{-\tau(x)}}  } + 1 \Big) \\
\leq&\ \frac{\sqrt{n}}{\beta}\Big( \frac{1}{   1  -\frac{\beta}{\kappa}  } + 1 \Big),
\end{align*}
where all three inequalities above follow from the fact that $x_2 e^{-\tau(x)} \geq \kappa/\sqrt{n} > \beta/\sqrt{n}$; c.f.\ \eqref{eq:gammaprop} in Lemma~\ref{lem:tau}. 
Taking the derivative in \eqref{eq:d1} with respect to $x_2$, we see that
\begin{align*}
&f_{12}^*(x)=  \frac{\sqrt{n}}{\beta} \big( -\tau_2(x) e^{-\tau(x)}\big) \big(   x_2  e^{-\tau(x)}   -\kappa/\sqrt{n}  \big) + \frac{\sqrt{n}}{\beta}  e^{-\tau(x)} \big( e^{-\tau(x)}   -  \tau_2(x) x_2  e^{-\tau(x)} \big). 
\end{align*}
The quantity above is non-negative because  $x_2 e^{-\tau(x)} \geq \kappa/\sqrt{n}$ and  $-\tau_2(x) \geq 0$; the latter follows from \eqref{eq:tauder}. Lastly, we can differentiate \eqref{eq:d2} and use $\tau_2(x) = \tau_1(x) \tau(x)$ from \eqref{eq:tauder} to see that 
\begin{align*}
f_{22}^*(x) =&\ \frac{1}{x_2^2} \frac{\kappa}{\sqrt{n}}    + \frac{\sqrt{n}}{\beta}\big(    e^{-\tau(x)} - \tau_1(x) \tau(x) x_2 e^{-\tau(x)}  \big) \Big( \frac{x_2  e^{-\tau(x)}   -\beta/\sqrt{n}}{x_2} + \tau(x) e^{-\tau(x)} \Big)\\
 &+ \frac{\sqrt{n}}{\beta}\big(   x_2  e^{-\tau(x)} - \kappa/\sqrt{n}\big) \Big( \frac{\beta/\sqrt{n}}{x_2^2}  - \tau_1(x) \tau^2(x) e^{-\tau(x)} \Big).
\end{align*}
Again, $f_{22}^*(x) \geq 0$ because  $x_2e^{-\tau(x)} \geq \kappa/\sqrt{n}$ and $-\tau_1(x) \geq 0$. Let us now bound $f_{22}^*(x)$. The first term on the right hand side above is bounded by $\sqrt{n}/\kappa$, because $x \geq \Gamma^{(\kappa)}$ implies $x_2 \geq \kappa/\sqrt{n}$. For the second term, note that 
\begin{align*}
 &\frac{x_2  e^{-\tau(x)}   -\beta/\sqrt{n}}{x_2} + \tau(x) e^{-\tau(x)}  \leq  1 + e^{-1} \leq 2,  
\end{align*}
and using the form of $\tau_1(x)$ from \eqref{eq:tauder},
\begin{align*}
 e^{-\tau(x)} - \tau_1(x) \tau(x) x_2 e^{-\tau(x)}   =&\   e^{-\tau(x)} +\frac{e^{-\tau(x)}}{x_2 e^{-\tau(x)}-\beta/\sqrt{n}} \tau x_2 e^{-\tau(x)} \\
 =&\ e^{-\tau(x)} + \frac{\tau e^{-\tau(x)}}{1-\frac{\beta/\sqrt{n}}{x_2e^{-\tau(x)}} } \leq 1 + \frac{e^{-1}}{1-\frac{\beta }{\kappa} }  \leq 1 + \frac{\kappa}{\kappa-\beta}.
\end{align*}
 For the third term, observe that
\begin{align*}
&\big(   x_2  e^{-\tau(x)} - \kappa/\sqrt{n}\big) \Big( \frac{\beta/\sqrt{n}}{x_2^2}  - \tau_1(x) \tau^2(x) e^{-\tau(x)} \Big)\\
 \leq&\ \big(   x_2  e^{-\tau(x)} - \kappa/\sqrt{n}\big) \Big(\frac{1}{x_2} \frac{\beta}{\kappa}  - \tau_1(x) \tau^2(x) e^{-\tau(x)} \Big)\\
=&\ \frac{ x_2  e^{-\tau(x)} - \kappa/\sqrt{n}}{x_2}  \frac{\beta}{\kappa} + \frac{x_2  e^{-\tau(x)} - \kappa/\sqrt{n}}{x_2  e^{-\tau(x)} - \beta/\sqrt{n}}    \tau^2(x) e^{-2\tau(x)}  \\
\leq&\ \frac{\beta}{\kappa} + 1 \leq 2,
\end{align*}
where we use $x_2 \geq \kappa/\sqrt{n}$ in the first inequality, the form of $\tau_1(x)$ in \eqref{eq:tauder} in the first equation, and the fact that $\beta < \kappa$ in the last two inequalities.
Combining the bounds on all three terms, we conclude that 
\begin{align*}
f_{22}^*(x) \leq \frac{\sqrt{n}}{\kappa} + \frac{\sqrt{n}}{\beta} 2 \Big( 1 + \frac{\kappa}{\kappa-\beta}\Big) + 2\frac{\sqrt{n}}{\beta} \leq \frac{\sqrt{n}}{\beta}  \Big( 5 + \frac{2\kappa}{\kappa-\beta}\Big),
\end{align*}
where in the last inequality we used the fact that $\sqrt{n}/\kappa < \sqrt{n}/\beta$.
\end{proof}

\section{Proving Lemma~\ref{lem:expergodderivs}. } \label{app:expergodderivs}
In this section we prove Lemma~\ref{lem:expergodderivs} by constructing solutions to \eqref{eq:pde1} and \eqref{eq:pde2}. The intuition behind the forms of these solutions is the same as in Section~\ref{sec:derbounds}. Namely, that
\begin{align*}
f^{(1)}(x) = \int_{0}^{\infty} \phi^{(\kappa_1/\sqrt{n},\kappa_2/\sqrt{n})}(-v_1^x(t)) dt, \quad \text{ and } \quad f^{(2)}(x) = \int_{0}^{\infty} \phi^{(\kappa_1/\sqrt{n},\kappa_2/\sqrt{n})}(v_2^x(t)) dt
\end{align*}
solves \eqref{eq:pde1} and \eqref{eq:pde2}, respectively. For the remainder of this section, we fix $\kappa_1   < \kappa_2$ with $\kappa_1 > \beta$, and let us write $\phi(x)$ instead of $\phi^{(\kappa_1/\sqrt{n},\kappa_2/\sqrt{n})}(x)$ to simplify notation.
\subsection{Solving the first PDE.}
We begin by constructing a candidate solution to \eqref{eq:pde1}, and proving the associated properties in Lemma~\ref{lem:expergodderivs}. Recall the definition of $W(x)$ from Section~\ref{sec:lambertw}, and for any $\kappa > \beta$, define 
\begin{align}
\tilde \tau^{(\kappa)}(x) =  \frac{-(x_1 + \beta/\sqrt{n})}{x_2} - W \Big( \frac{(\kappa-\beta)/\sqrt{n}}{x_2} e^{\frac{-(x_1+\beta/\sqrt{n})}{x_2} } \Big), \quad x_1 \leq -\kappa/\sqrt{n},\ x_2 > 0. \label{eq:tautilde}
\end{align}
The quantity in \eqref{eq:tautilde} is well defined because the argument of $W(\cdot)$ is positive. Furthermore, differentiability of  $W(\cdot)$ implies  differentiability of $\tilde \tau^{(\kappa)}(x)$.  One can check that
\begin{align}
-\beta/\sqrt{n} + (x_1 + \beta/\sqrt{n}) e^{-\tilde \tau^{(\kappa)}(x)} + x_2 \tilde \tau^{(\kappa)}(x) e^{-\tilde \tau^{(\kappa)}(x)} = -\kappa/\sqrt{n} \label{eq:tauhit}
\end{align}
for those $x$ where $\tilde \tau^{(\kappa)}(x)$ is defined by repeating the arguments used to show the equivalence of \eqref{eq:step1} and \eqref{eq:eta} in Section~\ref{app:gamma}.  Intuitively,  $\tilde \tau^{(\kappa)}(x)$ is the time the fluid model hits the set $\{x_1 = -\kappa/\sqrt{n}\}$. The following lemma tells us that we can extend the definition of $\tilde \tau^{(\kappa)}(x)$ to $x_2 = 0$; it is proved in Section~\ref{app:taulim}. 
\begin{lemma} \label{lem:taulim}
For any $\kappa > \beta$ and  $x_1 \leq -\kappa/\sqrt{n}$, 
\begin{align}
\lim_{x_2 \downarrow 0} \tilde \tau^{(\kappa)}(x) = \log\Big( \frac{-\sqrt{n}x_1 - \beta}{\kappa-\beta} \Big), \label{eq:tauextendlimit}
\end{align}
meaning that the function in \eqref{eq:tautilde} can be extended to $x_2 \geq 0$. Furthermore,
\begin{align}
\tilde \tau^{(\kappa)}(-\kappa/\sqrt{n},x_2) = 0, \quad x_2 \geq 0 \label{eq:tautildezero}
\end{align}
and
\begin{align}
-\tilde \tau_1^{(\kappa)}(x) = \frac{e^{-\tilde \tau^{(\kappa)}(x)}}{x_2e^{-\tilde \tau^{(\kappa)}(x)} + (\kappa-\beta)/\sqrt{n}}, \quad \tilde \tau_2^{(\kappa)}(x) = -\tilde \tau_1^{(\kappa)}(x)  \tilde \tau^{(\kappa)}(x) \label{eq:tautildeder}
\end{align}
for $x_1 \leq -\kappa/\sqrt{n}$ and $x_2 \geq 0,$
where the derivatives at $\{x_1 = -\kappa/\sqrt{n}\}$ and $\{x_2 = 0\}$ are interpreted as the one-sided derivatives.
\end{lemma}
The following lemma presents the candidate solution to \eqref{eq:pde1}; it is proved in Section~\ref{app:phi1}. 
\begin{lemma} \label{lem:phi1}
The function $f^{(1)}: \Omega \to \R_+$ defined as  
\begin{align*}
f^{(1)}(x) =
\begin{cases}
\tilde  \tau^{(\kappa_2)}(x) + \int_{\tilde  \tau^{(\kappa_2)}(x)}^{\tilde \tau^{(\kappa_1)}(x)} \phi\big(\beta/\sqrt{n} -(x_1+\beta/\sqrt{n})e^{-t} - x_2 t e^{-t} \big) dt,& x_1 \leq -\kappa_2/\sqrt{n},\\
\int_{0}^{\tilde \tau^{(\kappa_1)}(x)} \phi\big(\beta/\sqrt{n} -(x_1+\beta/\sqrt{n})e^{-t} - x_2 t e^{-t} \big) dt,\quad & x_1 \in \Big[ -\frac{\kappa_2}{\sqrt{n}}, -\frac{\kappa_1}{\sqrt{n}} \Big], \\ 
0,  & x_1 \in [-\kappa_1/\sqrt{n},0].
\end{cases}
\end{align*}
belongs to $C^2(\Omega)$. Furthermore,
\begin{align*}
f_{1}^{(1)}(x) =
\begin{cases}
 - \int_{\tilde  \tau^{(\kappa_2)}(x)}^{\tilde \tau^{(\kappa_1)}(x)} e^{-t}\phi'\big(\beta/\sqrt{n} -(x_1+\beta/\sqrt{n})e^{-t} - x_2 t e^{-t} \big) dt,\quad &  x_1 \leq -\kappa_2/\sqrt{n},\\
- \int_{0}^{\tilde \tau^{(\kappa_1)}(x)} e^{-t}\phi'\big(\beta/\sqrt{n} -(x_1+\beta/\sqrt{n})e^{-t} - x_2 t e^{-t} \big) dt,& x_1 \in \Big[ -\frac{\kappa_2}{\sqrt{n}}, -\frac{\kappa_1}{\sqrt{n}} \Big],\\
0, &  x_1 \in [-\kappa_1/\sqrt{n}, 0].
\end{cases}
\end{align*}
and 
\begin{align*}
f_2^{(1)}(x) =
\begin{cases}
 - \int_{\tilde  \tau^{(\kappa_2)}(x)}^{\tilde \tau^{(\kappa_1)}(x)} te^{-t}\phi'\big(\beta/\sqrt{n} -(x_1+\beta/\sqrt{n})e^{-t} - x_2 t e^{-t} \big) dt,\quad &  x_1 \leq -\kappa_2/\sqrt{n},\\
- \int_{0}^{\tilde \tau^{(\kappa_1)}(x)} te^{-t}\phi'\big(\beta/\sqrt{n} -(x_1+\beta/\sqrt{n})e^{-t} - x_2 t e^{-t} \big) dt,& x_1 \in \Big[ -\frac{\kappa_2}{\sqrt{n}}, -\frac{\kappa_1}{\sqrt{n}} \Big],\\
0, &  x_1 \in [-\kappa_1/\sqrt{n}, 0].
\end{cases}
\end{align*}

\end{lemma}
Let us now verify that $f^{(1)}(x)$ from Lemma~\ref{lem:phi1} satisfies \eqref{eq:pde1}. The boundary condition $f_1^{(1)}(0,x_2) = f_2^{(1)}(0,x_2)$ is trivially satisfied. For $x_1 \leq -\kappa_2/\sqrt{n}$, 
\begin{align*}
&(-x_1 + x_2 - \beta/\sqrt{n}) f_1^{(1)}(x) - x_2 f_2^{(1)}(x)\\
 =&\ \int_{\tilde  \tau^{(\kappa_2)}(x)}^{\tilde \tau^{(\kappa_1)}(x)} \big((x_1 - x_2 + \beta/\sqrt{n}) e^{-t} + x_2 t e^{-t}\big) \phi'\big(\beta/\sqrt{n} -(x_1+\beta/\sqrt{n})e^{-t} - x_2 t e^{-t} \big) dt\\
=&\ \int_{\kappa_2/\sqrt{n}}^{\kappa_1/\sqrt{n}}  \phi'(u) du = \phi(\kappa_1/\sqrt{n}) - \phi(\kappa_2/\sqrt{n}) = -1 = -\phi(-x_1),
\end{align*}
and a similar argument works when $x_1 \in \Big[ -\frac{\kappa_2}{\sqrt{n}}, -\frac{\kappa_1}{\sqrt{n}} \Big]$. Therefore, $f^{(1)}(x)$ solves \eqref{eq:pde1}. 

We now bound $f_1^{(1)}(x)$ and $f_{11}^{(1)}(x)$ to prove \eqref{eq:expf1}, and then bound $f^{(1)}(x)$ to prove \eqref{eq:expfsum1}. Since $\phi'(\kappa_2/\sqrt{n}) = \phi'(\kappa_1/\sqrt{n}) = 0$ (see \eqref{eq:phiderzero}), differentiating $f_{1}^{(1)}(x)$ gives us
\begin{align*}
f_{11}^{(1)}(x) =
\begin{cases}
 - \int_{\tilde  \tau^{(\kappa_2)}(x)}^{\tilde \tau^{(\kappa_1)}(x)} e^{-2t}\phi''\big(\beta/\sqrt{n} -(x_1+\beta/\sqrt{n})e^{-t} - x_2 t e^{-t} \big) dt,\quad &  x_1 \leq -\kappa_2/\sqrt{n},\\
- \int_{0}^{\tilde \tau^{(\kappa_1)}(x)} e^{-2t}\phi''\big(\beta/\sqrt{n} -(x_1+\beta/\sqrt{n})e^{-t} - x_2 t e^{-t} \big) dt,& x_1 \in \Big[ -\frac{\kappa_2}{\sqrt{n}}, -\frac{\kappa_1}{\sqrt{n}} \Big],\\
0, &  x_1 \in [-\kappa_1/\sqrt{n}, 0].
\end{cases}
\end{align*}
 Using \eqref{eq:phider}, 
\begin{align*}
&\abs{f_1^{(1)}(x)} \leq \frac{4\sqrt{n}\abs{ \tilde \tau^{(\kappa_1)}(x) - \tilde  \tau^{(\kappa_2)}(x)}}{\kappa_2 - \kappa_1}, \quad &x_1 \leq -\kappa_2/\sqrt{n}, \\
&\abs{f_{11}^{(1)}(x)} \leq \frac{12n\abs{ \tilde \tau^{(\kappa_1)}(x) - \tilde  \tau^{(\kappa_2)}(x)}}{(\kappa_2 - \kappa_1)^2}, \quad &x_1 \leq -\kappa_2/\sqrt{n},\\
&\abs{f_1^{(1)}(x)} \leq \frac{4\sqrt{n}\abs{ \tilde \tau^{(\kappa_1)}(x)  }}{\kappa_2 - \kappa_1}, \quad \abs{f_{11}^{(1)}(x)} \leq \frac{12n\abs{ \tilde \tau^{(\kappa_1)}(x)}}{(\kappa_2 - \kappa_1)^2}, \quad &x_1 \in \Big[ -\frac{\kappa_2}{\sqrt{n}}, -\frac{\kappa_1}{\sqrt{n}} \Big].
\end{align*}
When $x_1 \leq -\kappa_2/\sqrt{n}$, 
\begin{align}
\tilde \tau^{(\kappa_1)}(x) - \tilde  \tau^{(\kappa_2)}(x) =&\ W \Big( \frac{(\kappa_2-\beta)/\sqrt{n}}{x_2} e^{\frac{-(x_1+\beta/\sqrt{n})}{x_2} } \Big) - W \Big( \frac{(\kappa_1-\beta)/\sqrt{n}}{x_2} e^{\frac{-(x_1+\beta/\sqrt{n})}{x_2} } \Big)\notag  \\
=&\ \log\Big( \frac{\kappa_2 -\beta}{\kappa_1-\beta} \Big)- \log \Bigg( \frac{W \Big( \frac{(\kappa_2-\beta)/\sqrt{n}}{x_2} e^{\frac{-(x_1+\beta/\sqrt{n})}{x_2} } \Big)}{W \Big( \frac{(\kappa_1-\beta)/\sqrt{n}}{x_2} e^{\frac{-(x_1+\beta/\sqrt{n})}{x_2} } \Big)}  \Bigg) \notag  \\
\leq&\ \log\Big( \frac{\kappa_2 -\beta}{\kappa_1-\beta} \Big), \quad x_1 \leq -\kappa_2/\sqrt{n}, \label{eq:tildetaudiff}
\end{align}
where the second equation follows from \eqref{eq:loglambert} and the inequality follows from the fact that $W(\cdot)$ is an increasing function and $\kappa_2 > \kappa_1$. The first equation in \eqref{eq:tildetaudiff} and monotonicity of $W(\cdot)$ means that  $\tilde \tau^{(\kappa_1)}(x) - \tilde  \tau^{(\kappa_2)}(x) > 0$, and therefore 
\begin{align*}
\abs{\tilde \tau^{(\kappa_1)}(x) - \tilde  \tau^{(\kappa_2)}(x)} \leq  \log\Big( \frac{\kappa_2 -\beta}{\kappa_1-\beta} \Big), \quad x_1 \leq -\kappa_2/\sqrt{n}.
\end{align*}
When $x_1 \in \Big[ -\frac{\kappa_2}{\sqrt{n}}, -\frac{\kappa_1}{\sqrt{n}} \Big]$,
\begin{align}
0 = \tilde \tau^{(\kappa_1)}(-\kappa_1/\sqrt{n},x_2) \leq&\ \tilde \tau^{(\kappa_1)}(x) \notag \\
\leq&\ \tilde \tau^{(\kappa_1)}(-\kappa_2/\sqrt{n},x_2)  \notag \\
=&\ \tilde \tau^{(\kappa_1)}(-\kappa_2/\sqrt{n},x_2) - \tilde \tau^{(\kappa_2)}(-\kappa_2/\sqrt{n},x_2) \notag \\
\leq&\  \log\Big( \frac{\kappa_2 -\beta}{\kappa_1-\beta} \Big). \label{eq:mid}
\end{align}
The two equalities  above are due to \eqref{eq:tautildezero}, the first two inequalities follow from the fact that $\tilde \tau_{1}^{(\kappa_1)}(x) \leq 0$ in \eqref{eq:tautildeder}, and the last inequality comes from \eqref{eq:tildetaudiff}. Therefore, 
\begin{align*}
&\abs{f_1^{(1)}(x)} \leq \frac{4\sqrt{n}}{\kappa_2-\kappa_1}\log\Big( \frac{\kappa_2 -\beta}{\kappa_1-\beta} \Big), \quad  \abs{f_{11}^{(1)}(x)} \leq \frac{12n}{ (\kappa_2-\kappa_1)^2}\log\Big( \frac{\kappa_2 -\beta}{\kappa_1-\beta} \Big).
\end{align*}
Choosing $\kappa_1 = \beta + \epsilon$ and $\kappa_2 = \beta + 2\epsilon$ proves \eqref{eq:expf1}. To prove \eqref{eq:expfsum1}, note that 
\begin{align*}
\abs{f^{(1)}(x)} \leq \tilde \tau^{(\kappa_1)}(x) \leq \log\Big( \frac{\kappa_2 -\beta}{\kappa_1-\beta} \Big), \quad x_1 \in \Big[ -\frac{\kappa_2}{\sqrt{n}}, -\frac{\kappa_1}{\sqrt{n}} \Big],
\end{align*}
where the first inequality follows from the form of $f^{(1)}(x)$ and the second inequality follows from \eqref{eq:mid}. For $x_1 \leq -\kappa_1/\sqrt{n}$, 
\begin{align*}
\abs{f^{(1)}(x)} \leq&\ \tilde \tau^{(\kappa_2)}(x) + \abs{\tilde \tau^{(\kappa_1)}(x) - \tilde  \tau^{(\kappa_2)}(x)} \\
=&\ \tilde \tau^{(\kappa_2)}(x) + \tilde \tau^{(\kappa_1)}(x) - \tilde  \tau^{(\kappa_2)}(x)\\
=&\ \tilde \tau^{(\kappa_1)}(x)\\
\leq&\ \log\Big( \frac{\kappa_2 -\beta}{\kappa_1-\beta} \Big),
\end{align*}
where the last inequality follows from \eqref{eq:tildetaudiff} and the fact that $\tilde \tau^{(\kappa_2)}(x) \geq 0$. Choosing $\kappa_1 = \beta + \epsilon$ and $\kappa_2 = \beta + 2\epsilon$ proves \eqref{eq:expfsum1}.

\subsubsection{Proof of Lemma~\ref{lem:taulim}. } \label{app:taulim}
\begin{proof}[Proof of Lemma~\ref{lem:taulim}] 
Let 
\begin{align*}
u = \frac{(\kappa-\beta)/\sqrt{n}}{x_2} e^{\frac{-(x_1+\beta/\sqrt{n})}{x_2} }.
\end{align*}
Using \eqref{eq:loglambert}, 
\begin{align*}
\tilde \tau^{(\kappa)}(x) =  \frac{-(x_1 + \beta/\sqrt{n})}{x_2} - W(u) =   \log\bigg( \frac{\sqrt{n}}{\kappa-\beta}   x_2W(u) \bigg), \quad x_1 \leq -\kappa/\sqrt{n},\ x_2 > 0.
\end{align*}
Therefore, it remains to evaluate 
\begin{align*}
\lim_{x_2 \downarrow 0} \frac{W\Big( \frac{(\kappa-\beta)/\sqrt{n}}{x_2} e^{\frac{-(x_1+\beta/\sqrt{n})}{x_2} } \Big)}{1/x_2},
\end{align*}
which we do using L'Hopital's rule. The derivative of the numerator with respect to $x_2$ is 
\begin{align*}
W'(u)u \bigg( -\frac{1}{x_2} + \frac{x_1+\beta/\sqrt{n}}{x_2^2} \bigg)=&\ \frac{W(u)}{1 + W(u)}  \frac{1}{x_2^2} ( -x_2 + x_1+\beta/\sqrt{n}),
\end{align*}
where we used \eqref{eq:wder} to get the equality above. Therefore 
\begin{align*}
\lim_{x_2 \downarrow 0} \frac{W\Big( \frac{(\kappa-\beta)/\sqrt{n}}{x_2} e^{\frac{-(x_1+\beta/\sqrt{n})}{x_2} } \Big)}{1/x_2} = \lim_{x_2 \downarrow 0} \frac{\frac{d}{dx_2} W\Big( \frac{(\kappa-\beta)/\sqrt{n}}{x_2} e^{\frac{-(x_1+\beta/\sqrt{n})}{x_2} } \Big)}{-1/x_2^2} = - x_1 - \beta/\sqrt{n},
\end{align*}
and \eqref{eq:tauextendlimit} follows.  We now prove \eqref{eq:tautildezero}, or that $\tilde \tau^{(\kappa)}(-\kappa/\sqrt{n}, x_2) = 0$ for $x_2 \geq 0$. The claim is true when $x_2 = 0$ by \eqref{eq:tauextendlimit}. For $x_2 > 0$,  
\begin{align*}
&\tilde \tau^{(\kappa)}(-\kappa/\sqrt{n}, x_2) = \frac{-(-\kappa/\sqrt{n} + \beta/\sqrt{n})}{x_2} - W \Big( \frac{(\kappa-\beta)/\sqrt{n}}{x_2} e^{\frac{-(-\kappa/\sqrt{n}+\beta/\sqrt{n})}{x_2} } \Big) = 0, 
\end{align*}
where the second equality comes from the fact that $W(x e^{x}) = x$; c.f. \eqref{eq:lambertprop}. That  Differentiability of $ \tilde \tau^{(\kappa)}(x)$ follows from the differentiability of $W(\cdot)$. To verify \eqref{eq:tautildeder}, plug in \eqref{eq:tautilde} into \eqref{eq:tauhit} to see that 
\begin{align*}
W\Big( \frac{(\kappa-\beta)/\sqrt{n}}{x_2} e^{\frac{-(x_1+\beta/\sqrt{n})}{x_2} } \Big) = \frac{(\kappa-\beta)/\sqrt{n}}{x_2 e^{-\tilde \tau^{(\kappa)}(x)}},
\end{align*}
and therefore 
\begin{align*}
\tilde \tau^{(\kappa)}(x) = \frac{-(x_1+\beta/\sqrt{n})}{x_2} - \frac{(\kappa-\beta)/\sqrt{n}}{x_2 e^{-\tilde \tau^{(\kappa)}(x)}}.
\end{align*}
The proof of \eqref{eq:tautildeder} is then identical to the arguments used to prove \eqref{eq:exampletau1}-\eqref{eq:exampletau2}.

\end{proof}

\subsubsection{Proof of Lemma~\ref{lem:phi1}. } \label{app:phi1}
\begin{proof}[Proof of Lemma~\ref{lem:phi1} ]
For convenience, we recall that 
\begin{align*}
&f^{(1)}(x) \\
=&\
\begin{cases}
\tilde  \tau^{(\kappa_2)}(x) + \int_{\tilde  \tau^{(\kappa_2)}(x)}^{\tilde \tau^{(\kappa_1)}(x)} \phi\big(\beta/\sqrt{n} -(x_1+\beta/\sqrt{n})e^{-t} - x_2 t e^{-t} \big) dt,\quad x_1 \leq -\kappa_2/\sqrt{n},& \\
\int_{0}^{\tilde \tau^{(\kappa_1)}(x)} \phi\big(\beta/\sqrt{n} -(x_1+\beta/\sqrt{n})e^{-t} - x_2 t e^{-t} \big) dt, \hfill  x_1 \in \Big[ -\frac{\kappa_2}{\sqrt{n}}, -\frac{\kappa_1}{\sqrt{n}} \Big],& \\ 
0, \hfill  x_1 \in [-\kappa_1/\sqrt{n},0].&
\end{cases}
\end{align*} 
Continuity of $f^{(1)}(x)$ on the sets $\{x_1 = -\kappa_2/\sqrt{n}\}$ and $\{x_1 = -\kappa_1/\sqrt{n}\}$ follows from \eqref{eq:tautildezero}. Let us differentiate $f^{(1)}(x)$ on the set $x_1 \leq -\kappa_2/\sqrt{n}$. Using the Leibniz rule, 
\begin{align*}
f_{1}^{(1)}(x) =&\ \tilde  \tau_{1}^{(\kappa_2)}(x) + \phi\big(\beta/\sqrt{n} -(x_1+\beta/\sqrt{n})e^{-\tilde  \tau^{(\kappa_1)}(x)} - x_2 \tilde  \tau^{(\kappa_1)}(x) e^{-\tilde  \tau^{(\kappa_1)}(x)} \big) \tilde \tau_{1}^{(\kappa_1)}(x) \\
&- \phi\big(\beta/\sqrt{n} -(x_1+\beta/\sqrt{n})e^{-\tilde  \tau^{(\kappa_2)}(x)} - x_2 \tilde  \tau^{(\kappa_2)}(x) e^{-\tilde  \tau^{(\kappa_2)}(x)} \big) \tilde  \tau_{1}^{(\kappa_2)}(x)\\
&- \int_{\tilde  \tau^{(\kappa_2)}(x)}^{\tilde \tau^{(\kappa_1)}(x)} e^{-t}\phi'\big(\beta/\sqrt{n} -(x_1+\beta/\sqrt{n})e^{-t} - x_2 t e^{-t} \big) dt\\
=&\ \tilde  \tau_{1}^{(\kappa_2)}(x) + \phi\big(\kappa_1/\sqrt{n}\big) \tilde \tau_{1}^{(\kappa_1)}(x) - \phi\big(\kappa_2/\sqrt{n}\big) \tilde  \tau_{1}^{(\kappa_2)}(x)\\
&- \int_{\tilde  \tau^{(\kappa_2)}(x)}^{\tilde \tau^{(\kappa_1)}(x)} e^{-t}\phi'\big(\beta/\sqrt{n} -(x_1+\beta/\sqrt{n})e^{-t} - x_2 t e^{-t} \big) dt\\
=&\ - \int_{\tilde  \tau^{(\kappa_2)}(x)}^{\tilde \tau^{(\kappa_1)}(x)} e^{-t}\phi'\big(\beta/\sqrt{n} -(x_1+\beta/\sqrt{n})e^{-t} - x_2 t e^{-t} \big) dt,
\end{align*}
where the second and third equalities follow from \eqref{eq:tauhit}, or 
\begin{align*}
&\phi\big(\beta/\sqrt{n} -(x_1+\beta/\sqrt{n})e^{-\tilde  \tau^{(\kappa_2)}(x)} - x_2 \tilde  \tau^{(\kappa_2)}(x) e^{-\tilde  \tau^{(\kappa_2)}(x)} \big) = \phi(\kappa_2/\sqrt{n}) = 1,\\
&\phi\big(\beta/\sqrt{n} -(x_1+\beta/\sqrt{n})e^{-\tilde  \tau^{(\kappa_1)}(x)} - x_2 \tilde  \tau^{(\kappa_1)}(x) e^{-\tilde  \tau^{(\kappa_1)}(x)} \big) = \phi(\kappa_1/\sqrt{n}) = 0.
\end{align*}
Repeating the same argument on the set $x_1 \in \big[ -\frac{\kappa_2}{\sqrt{n}}, -\frac{\kappa_1}{\sqrt{n}} \big]$, we conclude that
\begin{align*}
&f_{1}^{(1)}(x) \\
=&
\begin{cases}
 - \int_{\tilde  \tau^{(\kappa_2)}(x)}^{\tilde \tau^{(\kappa_1)}(x)} e^{-t}\phi'\big(\beta/\sqrt{n} -(x_1+\beta/\sqrt{n})e^{-t} - x_2 t e^{-t} \big) dt,\quad &  x_1 \leq -\kappa_2/\sqrt{n},\\
- \int_{0}^{\tilde \tau^{(\kappa_1)}(x)} e^{-t}\phi'\big(\beta/\sqrt{n} -(x_1+\beta/\sqrt{n})e^{-t} - x_2 t e^{-t} \big) dt,& x_1 \in \Big[ -\frac{\kappa_2}{\sqrt{n}}, -\frac{\kappa_1}{\sqrt{n}} \Big],\\
0, &  x_1 \in [-\kappa_1/\sqrt{n}, 0].
\end{cases}
\end{align*} 
and 
\begin{align*}
&f_{2}^{(1)}(x) \\
=&
\begin{cases}
 - \int_{\tilde  \tau^{(\kappa_2)}(x)}^{\tilde \tau^{(\kappa_1)}(x)} te^{-t}\phi'\big(\beta/\sqrt{n} -(x_1+\beta/\sqrt{n})e^{-t} - x_2 t e^{-t} \big) dt,\quad &  x_1 \leq -\kappa_2/\sqrt{n},\\
- \int_{0}^{\tilde \tau^{(\kappa_1)}(x)} te^{-t}\phi'\big(\beta/\sqrt{n} -(x_1+\beta/\sqrt{n})e^{-t} - x_2 t e^{-t} \big) dt,& x_1 \in \Big[ -\frac{\kappa_2}{\sqrt{n}}, -\frac{\kappa_1}{\sqrt{n}} \Big],\\
0, &  x_1 \in [-\kappa_1/\sqrt{n}, 0].
\end{cases}
\end{align*}
Continuity of the first  order derivatives follows from \eqref{eq:tautildezero}. Furthermore, existence and continuity of $f_{11}^{(1)}(x), f_{12}^{(1)}(x)$, and $f_{22}^{(1)}(x)$ follows from existence and continuity of $\phi''(x)$. 
\end{proof}

\subsection{Solving the second PDE.}
Recall the definitions of $\Gamma^{(\kappa)}$ and $\tau(x)$ from Lemmas~\ref{lem:gamma} and \ref{lem:tau}, respectively. Partition $\Omega$ into four subdomains:
\begin{align*}
&S_0 =  \{x \in \Omega \ |\ x_2 \leq \kappa_1/\sqrt{n} \},\\
&S_1 =  \{x \in \Omega \ |\ x_2 \geq \kappa_1/\sqrt{n},\ x \leq \Gamma^{(\kappa_1)} \},\\
&S_2 = \{x \in \Omega \ |\   \Gamma^{(\kappa_1)} \leq x \leq \Gamma^{(\kappa_2)} \},\\
&S_3 =  \{x \in \Omega \ |\ x \geq \Gamma^{(\kappa_2)}\}.
\end{align*}
Figure~\ref{fig:visualization} helps to visualize the four sets. Let us verify that $S_0\cup S_1 \cup S_2 \cup S_3 $ does indeed equal $\Omega$. Fix $x_1 \leq 0$ and $x_2 \geq 0$. Then $x$ must lie in one of 
\begin{align*}
&\{x_2 \leq \kappa_1/\sqrt{n}\}, \quad  \{ x_2 \geq \kappa_1/\sqrt{n},\ x \leq \Gamma^{(\kappa_1)} \}, \quad   \{ x_2 \geq \kappa_1/\sqrt{n},\    \Gamma^{(\kappa_1)} \leq x \leq \Gamma^{(\kappa_2)} \},\\
&  \text{ or } \{x_2 \geq \kappa_1/\sqrt{n},\ x \geq \Gamma^{(\kappa_1)},\  x \geq \Gamma^{(\kappa_2)}\}.
\end{align*}  
From \eqref{eq:gammaprop} we see that $x \geq \Gamma^{(\kappa_1)}$ implies  $x_2 \geq \kappa_1/\sqrt{n}$, or
\begin{align*}
 &\{ x_2 \geq \kappa_1/\sqrt{n},\    \Gamma^{(\kappa_1)} \leq x \leq \Gamma^{(\kappa_2)} \} = S_2, \quad \text{ and } \\
 &\{x_2 \geq \kappa_1/\sqrt{n},\ x \geq \Gamma^{(\kappa_1)},\  x \geq \Gamma^{(\kappa_2)}\} = \{ x \geq \Gamma^{(\kappa_1)},\  x \geq \Gamma^{(\kappa_2)}\}.
\end{align*}
From \eqref{eq:gammaabove} we know that $x \geq \Gamma^{(\kappa_2)}$ implies  $x > \Gamma^{(\kappa_1)}$, or 
\begin{align*}
\{ x \geq \Gamma^{(\kappa_1)},\  x \geq \Gamma^{(\kappa_2)}\} = S_3.
\end{align*}
Therefore $S_0\cup S_1 \cup S_2 \cup S_3 $  equals $\Omega$.
\begin{figure}
\begin{tikzpicture}
 \node (ref) at (0,0){};
 \draw [->] (12,1)->(2,1);
 \draw [->] (11,0)->(11,8);
\node at (11.8,3) { $\kappa_1/\sqrt{n}$};
\node at (11.8,4) { $\kappa_2/\sqrt{n}$};
\node at (1.5,6) { $\Gamma^{(\kappa_1)}$};
\node at (1.5,7) { $\Gamma^{(\kappa_2)}$};
\node at (11.8,8) { $x_2$};
\node at (2,0.5) { $x_1$};

\draw[dashed,red] (11,3) .. controls (9,4) and (5,5.8) .. (2,6); 
\draw[dashed,red] (11,4) .. controls (9,5) and (5,6.8) .. (2,7); 
\draw[dashed] (2,3)--(11,3);
\draw[dashed] (2,4)--(11,4);
 \end{tikzpicture}
 \caption{An aide to visualize $S_0,S_1,S_2,S_3$.} \label{fig:visualization}
\end{figure}
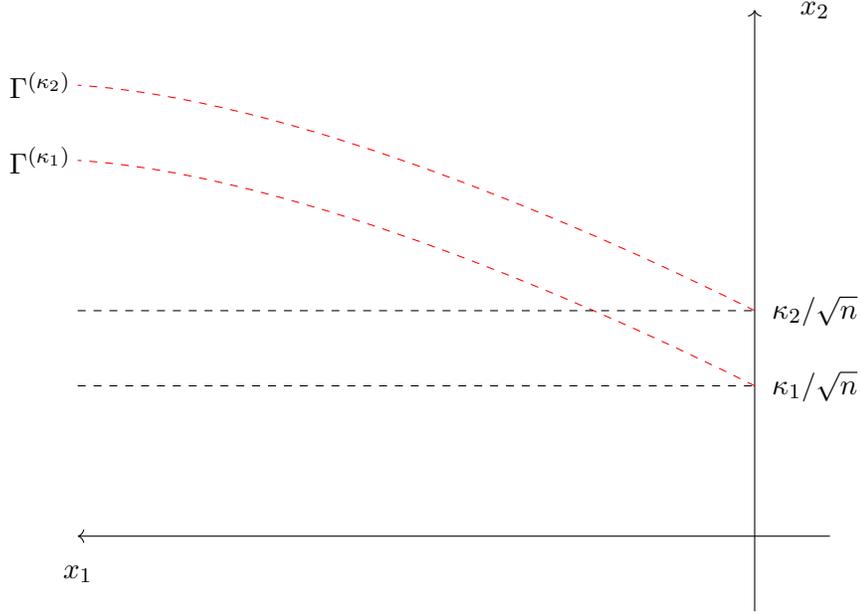
The following lemma is proved at the end of this section.
\begin{lemma} \label{lem:phi2}
The function 
\begin{align*}
f^{(2)}(x) = \begin{cases}
0, \hfill x \in S_0, \\
  \int_{0}^{\log(\sqrt{n}x_2/\kappa_1) } \phi( x_2 e^{-t}  ) dt,\ \hfill x_2 \leq \kappa_2/\sqrt{n},\   x \in S_1,\\
\log(\sqrt{n}x_2/\kappa_2) + \int_{0}^{\log(\kappa_2/\kappa_1) } \phi\big( \frac{\kappa_2}{\sqrt{n}}e^{-t} \big) dt,\ \qquad x_2 \geq \kappa_2/\sqrt{n},\   x \in S_1,\\
 \int_{0}^{\tau(x)  } \phi(x_2e^{-t})dt + \frac{\sqrt{n}}{\beta} \int_{\kappa_1/\sqrt{n}}^{x_2 e^{-\tau(x)}} \phi(t)dt,\ \hfill x_2 \leq \kappa_2/\sqrt{n},\ x \in S_2, \\
\log(\sqrt{n}x_2/\kappa_2) + \int_{\log(\sqrt{n}x_2/\kappa_2)}^{\tau(x)  } \phi \big( x_2  e^{-t} \big)dt  \\
 \hspace{3cm}+ \frac{\sqrt{n}}{\beta} \int_{\kappa_1/\sqrt{n}}^{x_2 e^{-\tau(x)}} \phi(t)dt, \hfill  x_2 \geq \kappa_2/\sqrt{n},\  x \in S_2, \\
\tau(x) + \frac{x_2 e^{-\tau(x)} - \kappa_2/\sqrt{n}}{\beta/\sqrt{n}} + \frac{\sqrt{n}}{\beta} \int_{\kappa_1/\sqrt{n}}^{\kappa_2/\sqrt{n}} \phi(t)dt, \hfill  x \in S_3,
\end{cases}
\end{align*}
is well-defined for $x \in \Omega$ and belongs to $C^{2}(\Omega)$. Its derivatives are
\begin{align}
f_{1}^{(2)}(x) = \begin{cases}
0, \quad &x \in S_0, \\ 
0,   &x \in S_1,\\
 \frac{\sqrt{n}}{\beta}  \phi(x_2e^{-\tau(x)})e^{-\tau(x)}, \quad & x \in S_2, \\ 
\frac{\sqrt{n}}{\beta} e^{-\tau(x)}, \quad &x \in S_3,
\end{cases} \label{eq:speciald1}
\end{align}
and 
\begin{align}
f_2^{(2)}(x) = \begin{cases}
0, \quad &x \in S_0, \\
\frac{1}{x_2} \phi(x_2) ,  &x \in S_1,\\
\frac{1}{x_2} \big( \phi(x_2) - \phi(x_2 e^{-\tau(x)})\big)  + \phi(x_2 e^{-\tau(x)})\frac{\sqrt{n}}{\beta}  e^{-\tau(x)} (\tau(x) + 1) ,& x \in S_2, \\ 
\frac{\sqrt{n}}{\beta}  e^{-\tau(x)} (\tau(x) + 1), \quad &x \in S_3.
\end{cases}\label{eq:speciald2}
\end{align}
\end{lemma}
In the remainder of this section, we 1) verify $L f^{(2)}(x) = -\phi (x_2)$, 2) verify $f_1^{(2)}(0,x_2) = f_2^{(2)}(0,x_2)$ 3) bound $f^{(2)}(x), f_{1}^{(2)}(x),$ and $f^{(2)}_{11}(x)$, and 4) prove Lemma~\ref{lem:phi2}.

We now verify that $L f^{(2)}(x) = -\phi (x_2)$. When $x \in S_0$ and $x \in S_1$, this fact is trivial. For $x \in S_2$,  
\begin{align*}
&(-x_1 + x_2 - \beta/\sqrt{n}) f_1^{(2)}(x) - x_2 f_2^{(2)}(x)\\
=&\ (-x_1 + x_2 - \beta/\sqrt{n}) \frac{\sqrt{n}}{\beta}  \phi(x_2e^{-\tau(x)})e^{-\tau(x)} \\
&-  x_2 \Big(\frac{1}{x_2} \big( \phi(x_2) - \phi(x_2 e^{-\tau(x)})\big)  + \phi(x_2 e^{-\tau(x)})\frac{\sqrt{n}}{\beta}  e^{-\tau(x)} (\tau(x) + 1)\Big) \\
=&\ -(x_1 +\beta/\sqrt{n}) \frac{\sqrt{n}}{\beta}  \phi(x_2e^{-\tau(x)})e^{-\tau(x)} \\
&-    \big( \phi(x_2) - \phi(x_2 e^{-\tau(x)})\big)  -x_2 \phi(x_2 e^{-\tau(x)})\frac{\sqrt{n}}{\beta}  e^{-\tau(x)} \tau(x)  \\
=&\  \frac{\sqrt{n}}{\beta}  \phi(x_2e^{-\tau(x)})e^{-\tau(x)}\Big( -(x_1 +\beta/\sqrt{n}) - x_2\tau(x) \Big) - \big( \phi(x_2) - \phi(x_2 e^{-\tau(x)})\big)  \\
=&\    \phi(x_2e^{-\tau(x)})  - \big( \phi(x_2) - \phi(x_2 e^{-\tau(x)})\big)  \\
 =&\ -\phi(x_2), \quad x \in \Omega,
\end{align*}
where in the second last inequality we used the fact that 
\begin{align*}
-(x_1+\beta/\sqrt{n})e^{-\tau(x)} - x_2 \tau(x) e^{-\tau(x)} = -\beta/\sqrt{n}
\end{align*}
from Lemma~\ref{lem:tau}. For $x \in S_3$, 
\begin{align*}
&(-x_1 + x_2 - \beta/\sqrt{n}) f_1^{(2)}(x) - x_2 f_2^{(2)}(x)\\
=&\ (-x_1 + x_2 - \beta/\sqrt{n})\frac{\sqrt{n}}{\beta} e^{-\tau(x)} - x_2 \frac{\sqrt{n}}{\beta}  e^{-\tau(x)} (\tau(x) + 1)\\
=&\ \frac{\sqrt{n}}{\beta} e^{-\tau(x)}\big(-x_1  - \beta/\sqrt{n}  - x_2  \tau(x) \big) \\
=&\ -1 = -\phi(x_2),
\end{align*}
where in the last equality we used the fact that $x \in S_3$ implies $x_2 \geq \kappa_2/\sqrt{n}$. Therefore, $L f^{(2)}(x) = -\phi (x_2)$ for all $x \in \Omega$. Let us now verify that 
\begin{align}
f_1^{(2)}(0,x_2) = f_2^{(2)}(0,x_2). \label{eq:obliqueverify}
\end{align}
From Lemma~\ref{lem:tau} we know that $\tau(0,x_2) = 0$, which  suggests that \eqref{eq:obliqueverify} holds for $x \in S_2 \cup	S_3$.  Furthermore, the only point in $S_1$ with $x_1 = 0$ is the point $(0,\kappa_1/\sqrt{n})$, which means that \eqref{eq:obliqueverify} holds for $x \in S_1$ as well. Therefore,  $f^{(2)}(x)$ solves \eqref{eq:pde2}.

We now bound $f_{1}^{(2)}(x)$ and $f^{(2)}_{11}(x)$ to prove \eqref{eq:expf2}. The bound on $f_{1}^{(2)}(x)$ is straightforward and the details are omitted. Differentiating $f_{1}^{(2)}(x)$, we see that
\begin{align*}
f_{11}^{(2)}(x) = \begin{cases}
0, \quad &x \in S_0, \\ 
0,   &x \in S_1,\\
 -\tau_1(x) \frac{\sqrt{n}}{\beta}  \phi(x_2e^{-\tau(x)})e^{-\tau(x)}  - \tau_1(x) x_2 e^{-2\tau(x)} \frac{\sqrt{n}}{\beta}  \phi'(x_2e^{-\tau(x)})  , \quad & x \in S_2, \\ 
-\tau_1(x) \frac{\sqrt{n}}{\beta}  e^{-\tau(x)}  , \quad &x \in S_3.
\end{cases}
\end{align*}
Recall from \eqref{eq:tauder} that
\begin{align*}
-\tau_1(x) = \frac{e^{-\tau(x)}}{x_2 e^{-\tau(x)} - \beta/\sqrt{n}}. 
\end{align*}
From \eqref{eq:gammaprop}, we know that $x_2 e^{-\tau(x)} \geq \kappa_2/\sqrt{n}$ for $x  \geq \Gamma^{(\kappa_2)}$, which means that
\begin{align*}
\abs{f_{11}^{(2)}(x)} \leq \frac{\sqrt{n}}{\beta} \frac{1}{\kappa_2/\sqrt{n}-\beta/\sqrt{n}} = \frac{n}{\beta(\kappa_2-\beta)}, \quad x \in S_3.
\end{align*}
Similarly,
\begin{align*}
\abs{f_{11}^{(2)}(x)} \leq&\ \frac{n}{\beta(\kappa_1-\beta)}+ \frac{e^{-\tau(x)}}{x_2 e^{-\tau(x)} - \beta/\sqrt{n}} x_2 e^{-2\tau(x)} \frac{\sqrt{n}}{\beta}  \abs{\phi'(x_2e^{-\tau(x)})} \\
\leq&\  \frac{n}{\beta(\kappa_1-\beta)} +  \frac{x_2 e^{-\tau(x)}}{x_2 e^{-\tau(x)} - \beta/\sqrt{n}} \frac{\sqrt{n}}{\beta}\frac{4\sqrt{n}}{\kappa_2-\kappa_1}  \\
\leq&\  \frac{n}{\beta(\kappa_1-\beta)} + \frac{\kappa_1}{\kappa_1	-\beta} \frac{4n}{\beta(\kappa_2-\kappa_1)}, \quad x \in S_2,
\end{align*}
where in the second inequality we used \eqref{eq:phider} and in the last inequality we used the fact that  $x_2 e^{-\tau(x)} \geq \kappa_1/\sqrt{n}$ for $x  \geq \Gamma^{(\kappa_1)}$. This proves \eqref{eq:expf2} and we now prove the bound on $f^{(2)}(x)$ in  \eqref{eq:expfsum2}.  From the form of $f^{(2)}(x)$, we see that 
\begin{align*}
&f^{(2)}(x) \leq \log(\sqrt{n}x_2/\kappa_1) \leq \log(\kappa_2/\kappa_1), \quad &x_2 \leq \kappa_2/\sqrt{n},\ x \in S_1,\\
&f^{(2)}(x) \leq \tau(x)  + \frac{\sqrt{n}}{\beta} (x_2 e^{-\tau(x)} - \kappa_1/\sqrt{n}) \leq \tau(x) + \frac{\kappa_2-\kappa_1}{\beta} ,\ \hfill &x_2 \leq \kappa_2/\sqrt{n},\ x \in S_2.
\end{align*}
We do not need to bound $f^{(2)}(x)$ for $x \in S_3$, because from \eqref{eq:gammaprop} we know that $x \in S_3$ implies $x_2 \geq \kappa_2/\sqrt{n}$. To bound $\tau(x)$ in the second line above, note from \eqref{eq:gammaprop} that 
\begin{align*}
x_2 e^{-\tau(x)} \geq  \kappa_1/\sqrt{n}, \quad x_2 \leq \kappa_2/\sqrt{n},\ x \in S_2,
\end{align*}
which means that
\begin{align*}
\tau(x) \leq \log(\kappa_2/\kappa_1), \quad x_2 \leq \kappa_2/\sqrt{n},\ x \in S_2.
\end{align*}
This proves \eqref{eq:expfsum2}.

\begin{proof}[Proof of Lemma~\ref{lem:phi2} ]
Recall for convenience that 
\begin{align*}
 f^{(2)}(x) = \begin{cases}
0, \hfill x \in S_0, \\
  \int_{0}^{\log(\sqrt{n}x_2/\kappa_1) } \phi( x_2 e^{-t}  ) dt,\ \hfill x_2 \leq \kappa_2/\sqrt{n}, \  x \in S_1,\\
\log(\sqrt{n}x_2/\kappa_2) + \int_{0}^{\log(\kappa_2/\kappa_1) } \phi\big( \frac{\kappa_2}{\sqrt{n}}e^{-t} \big) dt,\ \qquad  x_2 \geq \kappa_2/\sqrt{n},\  x \in S_1,\\
 \int_{0}^{\tau(x)  } \phi(x_2e^{-t})dt + \frac{\sqrt{n}}{\beta} \int_{\kappa_1/\sqrt{n}}^{x_2 e^{-\tau(x)}} \phi(t)dt,\ \hfill x_2 \leq \kappa_2/\sqrt{n},\ x \in S_2, \\
\log(\sqrt{n}x_2/\kappa_2) + \int_{\log(\sqrt{n}x_2/\kappa_2)}^{\tau(x)  } \phi \big( x_2  e^{-t} \big)dt  \\
\hspace{3cm}+ \frac{\sqrt{n}}{\beta} \int_{\kappa_1/\sqrt{n}}^{x_2 e^{-\tau(x)}} \phi(t)dt, \hfill x_2 \geq \kappa_2/\sqrt{n},\  x \in S_2, \\
\tau(x) + \frac{x_2 e^{-\tau(x)} - \kappa_2/\sqrt{n}}{\beta/\sqrt{n}} + \frac{\sqrt{n}}{\beta} \int_{\kappa_1/\sqrt{n}}^{\kappa_2/\sqrt{n}} \phi(t)dt, \hfill x \in S_3,
\end{cases}
\end{align*}
and 
\begin{align*}
&S_0 =  \{x \in \Omega \ |\ x_2 \leq \kappa_1/\sqrt{n} \},\\
&S_1 =  \{x \in \Omega \ |\ x_2 \geq \kappa_1/\sqrt{n},\ x \leq \Gamma^{(\kappa_1)} \},\\
&S_2 = \{x \in \Omega \ |\   \Gamma^{(\kappa_1)} \leq x \leq \Gamma^{(\kappa_2)} \},\\
&S_3 =  \{x \in \Omega \ |\ x \geq \Gamma^{(\kappa_2)}\}.
\end{align*}
Let us first verify the continuity of $f^{(2)}(x)$ on each of $S_i \cap S_j$ for $i,j \in \{0,1,2,3\}$; the only non-empty intersections are $S_0 \cap S_1$, $S_1 \cap S_2$, and $S_2 \cap S_3$. Continuity on $S_0 \cap S_1$ is straightforward because
\begin{align*}
S_0 \cap S_1 = \{x_2 = \kappa_1/\sqrt{n}\}.
\end{align*}
 Now $S_1 \cap S_2 = \{ x \in \Gamma^{(\kappa_1)}\}$, and recall from \eqref{eq:gammaprop} that
\begin{align}
x_2 e^{-\tau(x)} = \kappa/\sqrt{n} \quad \text{ or } \quad  \tau(x) = \log(\sqrt{n}x_2/\kappa), \quad x \in \Gamma^{(\kappa)}. \label{eq:auxprop}
\end{align}
Therefore, we can compare the definitions of $f^{(2)}(x)$ on 
\begin{align*}
x_2 \leq \kappa_2/\sqrt{n},\ x \in S_1  \quad \text{ vs. } \quad x_2 \leq \kappa_2/\sqrt{n},\ x \in S_2
\end{align*}
 and 
\begin{align*}
x_2 \geq \kappa_2/\sqrt{n},\ x \in S_1  \quad \text{ vs. } \quad x_2 \geq \kappa_2/\sqrt{n},\ x \in S_2
\end{align*}
to see that they coincide. Lastly, $S_2 \cap S_3  = \{x \in \Gamma^{(\kappa_2)}\}$, and recall that $x \in \Gamma^{(\kappa_2)}$ implies $x_2 \geq \kappa_2/\sqrt{n}$. Therefore, we need only to compare the definitions of $f^{(2)}(x)$ on 
\begin{align*}
x_2 \geq \kappa_2/\sqrt{n},\ x \in S_2  \quad \text{ vs. } \quad  x \in S_3
\end{align*}
and use \eqref{eq:auxprop} to conclude that $f^{(2)}(x)$ is continuous. 

Differentiating $f^{(2)}(x)$ and using the Leibniz integration rule, we get 
\begin{align*}
f_{1}^{(2)}(x) = \begin{cases}
0, \quad &x \in S_0, \\ 
0,   &x \in S_1,\\
 \tau_1(x) \phi(x_2e^{-\tau(x)})\big(1  - \frac{\sqrt{n}}{\beta} x_2 e^{-\tau(x)}\big), \quad & x \in S_2, \\ 
\tau_1(x) \big(1  - \frac{\sqrt{n}}{\beta} x_2 e^{-\tau(x)}\big), \quad &x \in S_3,
\end{cases}
\end{align*}
and
\begin{align*}
f_2^{(2)}(x) = \begin{cases}
0, \quad &x \in S_0, \\
\frac{1}{x_2} \phi(x_2) ,  &x \in S_1,\\
\frac{1}{x_2} \big( \phi(x_2) - \phi(x_2 e^{-\tau(x)})\big)  +  \tau_2(x) \phi(x_2 e^{-\tau(x)})  \\
 \hspace{1cm} + \frac{\sqrt{n}}{\beta}\phi(x_2 e^{-\tau(x)}) \big( e^{-\tau(x)} - \tau_2(x) x_2 e^{-\tau(x)}\big) ,& x \in S_2, \\ 
\tau_2(x) + \frac{\sqrt{n}}{\beta} \big(e^{-\tau(x)}- \tau_2(x) x_2 e^{-\tau(x)} \big), \quad &x \in S_3.
\end{cases}
\end{align*} 
From \eqref{eq:tauder}, we know that 
\begin{align*}
-\tau_1(x) = \frac{e^{-\tau(x)}}{x_2 e^{-\tau(x)} - \beta/\sqrt{n}}, \quad \text{ and } \quad \tau_2(x) = \tau_1(x) \tau(x),
\end{align*}
which proves the forms of $f^{(2)}_1(x)$ and $f^{(2)}_2(x)$ as stated in the Lemma. Let us now prove that the first order derivatives are continuous. We begin with $f_{1}^{(2)}(x)$, for which we have to verify continuity on $S_1 \cap S_2 = \{ x \in \Gamma^{(\kappa_1)}\}$ and $S_2 \cap S_3  = \{x \in \Gamma^{(\kappa_2)}\}$. On the former set, $\phi(x_2e^{-\tau(x)}) = \phi(\kappa_1/\sqrt{n}) = 0$, and on the latter set, $\phi(x_2e^{-\tau(x)}) = \phi(\kappa_2/\sqrt{n}) = 1$, which proves continuity of $f_{1}^{(2)}(x)$. We now prove continuity of $f^{(2)}_2(x)$. On $S_0 \cap S_1 = \{x_2 = \kappa_1/\sqrt{n}\}$, $f^{(2)}_2(x) = \frac{1}{x_2}\phi(\kappa_1/\sqrt{n}) =0$. On $S_1 \cap S_2 = \{ x \in \Gamma^{(\kappa_1)}\}$, continuity follows from the fact that $\phi(x_2e^{-\tau(x)}) = \phi(\kappa_1/\sqrt{n}) = 0$ and on $S_2 \cap S_3  = \{x \in \Gamma^{(\kappa_2)}\}$ continuity follows both from the fact that both $\phi(x_2e^{-\tau(x)}) = \phi(\kappa_2/\sqrt{n}) = 1$ and $1 \geq \phi(x_2)\geq \phi(x_2e^{-\tau(x)})=1$. Continuity of the second order partial derivatives is straightforward to check using similar arguments, and the proof is omitted.



\end{proof}

%
%
%


\bibliography{dai20171117} 

\end{document}